\documentclass[10pt]{amsart}
\usepackage{amscd}
\usepackage{amssymb}
\usepackage{latexsym}
\usepackage{graphicx}
\usepackage[parfill]{parskip} 
\usepackage{amsmath}
\usepackage{amsfonts}
\usepackage{amssymb}
\usepackage{latexsym}
\usepackage{graphicx}
\usepackage{enumerate}
\usepackage{bbm}

\usepackage[utf8]{inputenc}

\usepackage[final]{hyperref}
\hypersetup{colorlinks=true, linkcolor=blue, anchorcolor=blue, citecolor=red, filecolor=blue, menucolor=blue, pagecolor=blue, urlcolor=blue}
\newcommand{\abs}[1]{\vert #1 \vert}

\newcommand{\norm}[1]{\left\Vert #1 \right\Vert}

\newcommand{\bignorm}[1]{\bigl\Vert #1 \bigr\Vert}

\newcommand{\C}{\mathbb{C}}

\newcommand{\R}{\mathbb{R}}

\newcommand{\angles}[1]{\langle #1 \rangle}
\newcommand{\bigangles}[1]{\big\langle #1 \big\rangle}

\DeclareMathOperator{\diag}{diag}

\newtheorem{theorem}{Theorem}

\newtheorem{lemma}{Lemma}

\theoremstyle{definition}

\theoremstyle{remark}
\newtheorem{remark}{Remark}

\numberwithin{equation}{section}
\title[Local well-posedness of YM below energy norm]{Local well-posedness of Yang-Mills equations in Lorenz gauge below the energy norm }

\author{Achenef Tesfahun}
\address{Universit\"{a}t Bielefeld, Fakult\"{a}t  
f\"{u}r
Mathematik, 
Postfach 10 01 31 , D-33 Bielefeld, Germany}
\email{achenef@math.uni-bielefeld.de}

\thanks{2000 Mathematics Subject Classifications. 35Q40; 35L70.
\\The author acknowledges the support from the German research foundation, Collaboration Research Center 701.}

\begin{document}

\begin{abstract} 
We prove that the Yang-Mills equations in the Lorenz gauge (YM-LG) is locally well-posed for data below the energy norm, 
in particular, we can take data for the gauge potential $A$ and 
the associated curvature $F$ 
in $H^s\times H^{s-1}$ and $H^r\times H^{r-1}$ for $s=(\frac67+,-\frac1{14} +)$, respectively. This extends a recent 
result by Selberg 
and the present author on the local well-posedness of YM-LG for finite energy data (specifically, 
for $(s, r)=(1-,  0)$).
We also prove unconditional uniqueness of the energy class solution, that is, uniqueness in the classical space $C([-T, T]; X_0)$, 
where $X_0$ is
the energy data space. The key ingredient in the proof is the fact that most bilinear terms 
in YM-LG contain null structure some of which uncovered in the present paper. 

\end{abstract}

\maketitle
\section{Introduction}
The aim of this paper is to prove local well-posedness of YM-LG for data below the energy norm. As a consequence, we 
show that the energy class solution constructed recently by Selberg and the present author is unconditionally unique. 

Let $\mathcal G$ be a compact Lie group and $\mathfrak g$ its Lie algebra. For simplicity, we shall assume 
$\mathcal G=SO(n, \R)$ (the group of orthogonal matrices of determinant one) or 
$\mathcal G=SU(n, \C)$ (the group of unitary matrices of determinant one). 
Then $\mathfrak g=so(n, \R)$ (the algebra of skew symmetric matrices) or $\mathfrak g=su(n, \C)$ 
(the algebra of trace-free skew hermitian matrices).

Given a $\mathfrak g$-valued 1-form $A$ on the Minkowski space-time $\R^{1+3}$, 
we denote by  $F=F^{(A)}$ the associated curvature $F=dA + [A,A]$. That is, given
 $$A_\alpha \colon \R^{1+3} \to \mathfrak g, $$
 we define $F_{\alpha\beta} =  F^{(A)}_{\alpha\beta}$ by
\begin{equation}\label{Curvature}
  F_{\alpha\beta} = \partial_\alpha A_\beta - \partial_\beta A_\alpha + [A_\alpha,A_\beta],
\end{equation}
where $\alpha,\beta \in \{0,1,2,3\}$.  

 In this set up, the Yang-Mills equations (YM) read
  \begin{align}
 \label{YM1}
 \partial^\alpha F_{\alpha\beta} + [A^\alpha,  F_{\alpha\beta}]=0 , \quad  \beta \in \{0,1,2,3\},  
  \end{align} 
  where we follow the convention that repeated upper/lower indices are implicitly summed over their range.
  Indices are raised and lowered using the Minkowski metric $\diag(-1,1,1,1)$ on $\R^{1+3}$. 
  Roman indices $i, j,k,\dots$ run over $1,2,3$ and Greek indices $\alpha, \beta, \gamma $ over $0,1,2,3$.  
  Points on  $\R^{1+3}$ are written $(x^0,x^1,x^2,x^3)$ with $t=x^0$,
  and $\partial^\alpha$ denotes the partial derivative with respect to $x^\alpha$. 
  We write $\partial_t=\partial_0$, $\nabla = (\partial_1,\partial_2,\partial_3)$, and $\partial=(\partial_t, \nabla)$.  

 The total energy for YM, at time $t$,  is given by
$$
  \mathcal E(t) =\sum_{0\le \alpha, \beta\le 3} \int_{\R^3} \abs{F_{\alpha \beta}(t,x)}^2  \, dx,
$$
and is conserved for a smooth solution decaying sufficiently fast at spatial infinity, 
i.e., $$ \mathcal E(t)= \mathcal E(0).$$ 

The equation \eqref{YM1} is invariant under the
 \textit {gauge transformation}
 \begin{equation}\label{GaugeTransform}
       A_\alpha\to A_\alpha' =U A_\alpha  U^{-1}-  (\partial_\alpha U )  U^{-1}
  \end{equation}
  for sufficiently smooth function $ U: \R^{1+3}\rightarrow G$.
Indeed, if we  denote $F'=F^{(A')}$ and $ D'_\alpha=D^{(A')}_\alpha$, where $D_\alpha=D^{(A)}_\alpha$ is the \textit{covariant derivative} operator associated to $A$ 
  given by $D_\alpha = \partial_\alpha + [A_\alpha, \cdot]$, then a simple calculation shows that 
 \begin{align*}
   F' &=U F U^{-1}, \quad   D'_\alpha F'= U [D_\alpha F]  U^{-1}.
\end{align*} 
This in turn implies 
$$
 D'^\alpha{F'_{\alpha\beta}}= \partial^\alpha F'_{\alpha\beta} + [A'^\alpha,  F'_{\alpha\beta}]=0
$$
which shows that \eqref{YM1} is invariant under the \textit {gauge transformation} \eqref{GaugeTransform}, i.e., 
if $(A, F)$ satisfies \eqref{YM1}, so does  $(A',F')$.  A solution is therefore a representative of
its equivalent class, and hence we may impose an additional gauge condition (on $A$). 
The most popular gauges are the
temporal gauge:  $A_0=0$,
the Coulomb gauge: $\partial^i A_i=0$ and
 the Lorenz gauge:  $\partial^\alpha A_\alpha=0$.

 In both temporal and Lorenz gauges, YM can be written as a system of nonlinear wave equations whereas in 
 Coulomb gauge it is expressed as a system of nonlinear wave equations coupled with an elliptic equation.  
 In temporal gauge, Segal \cite{Segal1979} proved local and global well-posedness for initial data (for $A$)
 in the Sobolev space \footnote{Here $H^s = (I-\Delta)^{-s/2} L^2(\R^3)$.} $H^s \times H^{s-1}$ with $s\ge 3$.  
 This was improved later by Eardley and Moncrief \cite{Eardley1982, Eardley1982b} to $s\ge 2$ for the more general Yang-Mills-Higgs
 equations using the conservation of energy.  
  To prove well-posedness for finite energy data (that is, $s=1$), however, 
 requires the bilinear terms to be null forms. In Coulomb gauge, Klainerman-Machedon \cite{Klainerman1995} 
 showed that these bilinear terms are in fact null forms and used this fact to prove global well-posedness of YM
  for finite energy data. This result was later extended for the more general Yang-Mills-Higgs
 equations by Keel \cite{Keel1997}. Also in the temporal gauge YM contains a partial null structure, and 
  Tao \cite{Tao2003} used this fact to prove local well-posedness for $s>3/4$, for data with small norm. 
 Oh \cite{Oh2012, Oh2012b} developed a new approach based on the \emph{the Yang-Mills heat flow} to recover 
 the finite energy well-posedness result of Klainerman-Machedon \cite{Klainerman1995}.  Local and global regularity properties of the YM and Maxwell-Klein-Gordon equations have also been studied in $1+4$ dimensions, which is the energy-critical case; see \cite{Klainerman1999, Rodnianski2004, Sterbenz2007, Krieger2009}.

 Recently, Selberg and the present author \cite{st2013} discovered 
 null structure in most of the bilinear terms 
in YM-LG, and subsequently proved 
 local well-posedness for finite energy data. This result was later extended for the more general Yang-Mills-Higgs
 equations by the present author \cite{t2014}.
 In the present paper, we uncover additional null structure in YM-LG and prove local well-posedness
 for data below the energy norm, in particular, we can take data for $A$ and $F$ in $H^s\times H^{s-1}$ 
 and  $H^r\times H^{r-1}$ for $(s,r)=(\frac67+,-\frac1{14}+),$\footnote{We use the notation $a\pm=a\pm \varepsilon$,
 where $\varepsilon$
is assumed to be a sufficiently small positive number.} 
respectively.
 On the other hand, the scaling critical regularity exponents are $(s_c, r_c)=(\frac12,-\frac12).$ 
 Thus, there is still a gap left between the critical regularity and our result, yet this is the first 
 large data well-posedness result for YM below the energy norm. 
 To improve on this one might need to uncover null structure in all of the bilinear terms in YM-LG. 
 
 In \cite{st2013} the authors proved the  existence of energy class local solution
 $$(A, \partial_t A, F, \partial_t F)\in C([-T, T]; X_0),$$
 where $X_0$ is
the energy data space, but uniqueness was known only in
 the contraction space of $X^{s,b}$-type, which is strictly smaller than the natural solution space $C([-T, T]; X_0)$.
 Here we show that uniqueness in fact holds in the latter space, that is, the energy class solution 
 is unconditional unique.
  To prove this we rely on Strichartz estimates and product estimates in the wave-Sobolev spaces $H^{s,b}$ 
  which is due to D'Ancona, Foschi and Selberg \cite{dfs2012}. Using an
idea of Zhou \cite{z2000} we iteratively improve the known regularity of the solution, until
we reach a space where uniqueness is known.

 Lorenz-gauge null structure was first discovered in \cite{Selberg2010a} for the Maxwell-Dirac equations, 
 and then for the Maxwell-Klein-Gordon equations \cite{Selberg2010b} (see also \cite{Selberg2013}).

\subsection{YM-LG as a system of nonlinear wave equations}
Expanding \eqref{YM1} in terms of the gauge potentials $\left\{A_\alpha\right\}$, we get the following system of second order PDE: 
 \begin{equation}\label{YM2}
  \square A_\beta = \partial_\beta\partial^\alpha A_\alpha- [\partial^\alpha A_\alpha, A_\beta] - [A^\alpha,\partial^\alpha A_\beta] - 
  [A^\alpha, F_{\alpha\beta}].
\end{equation}
If we now impose the \emph{Lorenz gauge} condition,
then the system \eqref{YM2} reduces to the nonlinear wave equation
 \begin{equation}\label{YM3}
  \square A_\beta = - [A^\alpha,\partial_\alpha A_\beta] -   [A^\alpha, F_{\alpha\beta}].
\end{equation}

In addition, regardless of the choice of gauge, $F$ satisfies the wave equation 
    \begin{equation}\label{YMF1}
 \begin{split}
   \square F_{\beta\gamma}&=-[A^\alpha,\partial_\alpha F_{\beta\gamma}] - \partial^\alpha[A_\alpha,F_{\beta\gamma}] - \left[A^\alpha,[A_\alpha,F_{\beta\gamma}]\right]
   \\
   & \quad - 2[F^{\alpha}_{{\;\;\,}\beta},F_{\gamma\alpha}].
   \end{split}
 \end{equation}
  Indeed,  this will follow if we apply apply $D^\alpha$ to the 
 Bianchi identity
$$
  D_\alpha F_{\beta\gamma} + D_\beta F_{\gamma\alpha} + D_\gamma F_{\alpha\beta} = 0
$$
and simplify the resulting expression using the commutation identity
$$
  D_\alpha D_\beta X - D_\beta D_\alpha X = [F_{\alpha\beta},X]
$$
and  \eqref{YM1}  (see e.g. \cite{st2013}).

Expanding the second and fourth terms in \eqref{YMF1}, and also imposing the Lorenz gauge, yields 
\begin{equation}\label{YMF2}
\begin{split}
       \square F_{\beta\gamma}&= - 2[A^\alpha,\partial_\alpha F_{\beta\gamma}]
      + 2[\partial_\gamma A^\alpha, \partial_\alpha A_\beta]
      - 2[\partial_\beta A^\alpha, \partial_\alpha A_\gamma]
      \\
      &\quad + 2[\partial^\alpha A_\beta , \partial_\alpha A_\gamma]
                 + 2[\partial_\beta A^\alpha, \partial_\gamma A_\alpha] - [A^\alpha,[A_\alpha,F_{\beta\gamma}]] 
                 \\
                 &\quad  +2[F_{\alpha\beta},[A^\alpha,A_\gamma]]- 2[F_{\alpha\gamma},[A^\alpha,A_\beta]]
                      - 2[[A^\alpha,A_\beta],[A_\alpha,A_\gamma]] .  
                         \end{split}    
\end{equation}

If we ignore the matrix commutator structure and other special structures, and cubic and quartic terms 
in the equations \eqref{YM3}, \eqref{YMF2}, then we would obtain a schematic system of the form
\begin{align*}
 \square u&= u\partial u + uv,
 \\
 \square v&= u\partial v + \partial u \partial u .
\end{align*}
By using Strichartz estimates one can show local well-posedness of this system for data 
$(u(0), \partial_t u(0)) \in H^{1+\varepsilon}\times H^{\varepsilon}$ and 
$(v(0), \partial_t v(0)) \in H^{\varepsilon}\times H^{-1+\varepsilon}$
(see e.g. \cite{p1993}).
Moreover, in view of the counter examples of Lindblad \cite{lbd1996} these results are sharp. 
However, if the bilinear terms are null forms we can do better and prove local well-posedness 
for data in the energy class which corresponds to $\varepsilon=0$. 
In fact, we can even go below energy and prove local-well posedness for some $\varepsilon < 0$ 
by using bilinear estimates
in $X^{s, b}$-spaces.
The standard null forms are given by
\begin{equation}\label{OrdNullforms}
\left\{
\begin{aligned}
Q_{0}(u,v)&=\partial_\alpha u \partial^\alpha v=-\partial_t u \partial_t v+\partial_i u \partial^j v,
\\
Q_{\alpha\beta}(u,v)&=\partial_\alpha u \partial_\beta v-\partial_\beta u \partial_\alpha v.
\end{aligned}
        \right.
\end{equation}

  In Lorenz gauge, 
 all the bilinear terms in \eqref{YM3}, \eqref{YMF2} except $ [A^\alpha, F_{\alpha\beta}]$ turn out to be null forms.
 For this reason, we lose regularity on the solution $A$ starting from finite energy data.  
However, $A$ is only a potential representing the electromagnetic field $F$.
The most interesting physical quantity here is $F$, and 
we do not lose regularity for these quantity starting from finite energy data. In \cite{st2013}
 these facts were enough to prove local well-posedness for finite energy data.
 
 Note on the other hand by expanding the last term in the right hand side of \eqref{YM3}, we could write 
 \begin{equation}\label{YM4}
  \square A_\beta = - 2[A^\alpha,\partial_\alpha A_\beta] + [A^\alpha,\partial_\beta A_\alpha] - 
  [A^\alpha, [A_\alpha,A_\beta]].
\end{equation}
The cubic term in this equation does not cause a problem, but the new bilinear term $ [A^\alpha,\partial_\beta A_\alpha]$ does unless it contains a null structure. In fact it is worse than the term $[A^\alpha, F_{\alpha\beta}]$, since $F$ has better regularity than $\partial A$, the reason being  $F$ solves a nonlinear wave equation with bilinear null from terms. 
So in \cite{st2013} it was crucial that equation  \eqref{YM3} (together with \eqref{YMF2}) 
is used instead of the expanded version \eqref{YM4} 
in order to prove finite energy local well-posedness.

In the present paper, we shall nevertheless show the term $ [A^\alpha,\partial_\beta A_\alpha]$ in \eqref{YM4} does also contain 
a partial null structure. So expanding $ [A^\alpha, F_{\alpha\beta}]$ and writing the equation for $A$ as in \eqref{YM4} 
has in fact an advantage.
 By using a \emph{div-curl} decomposition of the spatial component of $A$, 
we shall write $ [A^\alpha,\partial_\beta A_\alpha]$ as a sum of  bilinear null form terms,
bilinear terms which are smoother, 
a bilinear term which contains only $F$ and higher order terms in $(A, F)$ which are harmless. 
The resulting (schematic) equation for $A$ will look like \footnote{Here we use the notation $\angles{\cdot} =\sqrt{1+\abs{\cdot}^2}$.}  
\begin{equation}\label{YM5}
\left\{
\begin{aligned}
\square A&=\Pi(A, \partial A)+ \Pi(\angles{\nabla}^{-1}A, A)+\Pi(\angles{\nabla}^{-2}A, \angles{\nabla}A)+ \Pi(\angles{\nabla}^{-1}F, F) 
\\
&\quad +  \Pi(A, A, A)+ \Pi(\angles{\nabla}^{-1}F, A, A)+\Pi(F, \angles{\nabla}^{-1}( A A))
\\
&\quad + \Pi(\angles{\nabla}^{-1}( AA), A,A),
 \end{aligned}
        \right.
\end{equation}
where $\Pi(\cdots)$ denotes a multilinear operator in its arguments. 

The cubic and quartic terms in \eqref{YM5} do not cause problems. We show that the first bilinear term $\Pi(A, \partial A)$ is a sum of null forms, 
whereas the second and the third bilinear terms are a lot smoother. 
The only term that causes difficulty is the bilinear term $\Pi(\angles{\nabla}^{-1}F, F) $ which as far as we know  is not a null form.
However, this term has at least better regularity than the term $[A^\alpha, F_{\alpha\beta}]$ in \eqref{YM3}, since $F$ solves a nonlinear wave equation with null form bilinear terms.
Using these facts we are able to prove local well-posedness for large data below the energy norm and consequently show that
the energy class solution is unconditionally unique.

\subsection{The Cauchy problem and statement of the result}
We want to solve the system \eqref{YMF2}-\eqref{YM4} simultaneously for $A$ and $F$.
So to pose the Cauchy problem for this system, we consider initial data for $(A,F)$ at $t=0$:
\begin{equation}\label{Data-AF}
\left\{
\begin{aligned}
    A(0) &= a, \quad \partial_t A(0) = \dot a,
        \\
    F(0) &=f, \quad \partial_t F(0) = \dot f.
 \end{aligned}
        \right.
\end{equation}
In fact, the initial data for $F$ can be determined from $(a, \dot a)$ as follows:
\begin{equation}\label{f}
\left\{
\begin{aligned}
  f_{ij} &= \partial_i a_j - \partial_j a_i + [a_i,a_j],
  \\
  f_{0i} &= \dot a_i - \partial_i a_0 + [a_0,a_i],
\\
  \dot f_{ij} &= \partial_i \dot a_j - \partial_j \dot a_i + [\dot a_i,a_j]+[ a_i, \dot a_j],
  \\
 \dot f_{0i} &= \partial^j f_{ji} +[a^\alpha, f_{\alpha i}] 
\end{aligned}
\right.
\end{equation}
where the first three expressions come from \eqref{Curvature} whereas 
the last one comes from \eqref{YM1} with $\beta=i$.

Note that the Lorenz gauge condition $\partial^\alpha A_\alpha=0$ and \eqref{YM1} with $\beta=0$ impose the constraints 
\begin{equation}\label{Const}
\left\{
\begin{aligned}
\dot a_0&= \partial^i a_i,
\\
   \partial^i f_{0i} &= [a^i, \dot a_{ i}].
\end{aligned}
\right.
\end{equation}
It turns out if $(A, F)$ is a solution to \eqref{YMF2}--\eqref{YM4} such that at $t=0$ \eqref{f}
and \eqref{Const} are satisfied, then the Lorenz gauge condition is satisfied for all times where 
the solution is sufficiently smooth.
That is, the Lorenz gauge condition propagates (see \cite{st2013} for the details). 

We now state our main result.
\begin{theorem}\label{MainThm}  Let $(s,r)=(\frac67+\varepsilon, -\frac1{14}+\varepsilon)$ or 
$(1-\varepsilon, 0)$ for sufficiently small $\varepsilon>0$.
\begin{enumerate}[(i)]
 \item (Local well-posedness.) Given initial data \eqref{Data-AF} with regularity
$$
 (a, \dot a) \in H^s\times H^{s-1} , \quad
 (f, \dot f) \in H^{r}\times H^{r-1},
$$
there exists a time $T> 0$ depending on the initial data norm,  
and a solution
\begin{equation}\label{Soln}
\left\{
\begin{aligned}
 A &\in C([-T,T]; H^{s}) \cap C^1([-T,T]; H^{s-1}),
  \\
  F &\in C([-T,T];  H^{r}) \cap C^1([-T,T]; H^{r-1}),
 \end{aligned}
        \right.
\end{equation}
solving the system \eqref{YMF2}--\eqref{YM4} on $S_T=(-T, T)\times \R^3$ in the sense of distributions.

The solution has the regularity
$$\left( A \pm \frac{1}{i\angles{\nabla}} \partial_t A \right)\in X_\pm^{s, \frac34+}(S_T), \quad 
\left( F \pm \frac{1}{i\angles{\nabla}} \partial_t F \right)\in X_\pm^{r, \frac12+}(S_T),$$
and it is the unique solution with this property (these spaces are defined in Section 3). 
Moreover, the solution depends continuously on the data
and higher regular data persists in time.
\item (Unconditional uniqueness of energy class solution.)
In the case where $(s,r)=(1-\varepsilon, 0)$,  the solution is in fact unique in the class \eqref{Soln}.
\end{enumerate}
\end{theorem}

\begin{remark}
Local well-posedness for $(s,r)=
(1-\varepsilon, 0)$ is proved in \cite{st2013}. It will be clear from the estimates that
one in fact obtain local well-posedness for $(s, r)$ in some convex region containing the points $(\frac67+\varepsilon, -\frac1{14}+\varepsilon)$ and 
$(1-\varepsilon, 0)$, but we do not pursue this here.
\end{remark}

Let us fix some notation. We use $\lesssim$ to mean $\le$ up to multiplication by 
a positive constant $C$ which may depend on $s$, $r$ and $T$. If $A,B $ are nonnegative quantities, 
$A \sim B$ means $B \lesssim A \lesssim B$.

The rest of the paper is organized as follows. In the next Section, 
we reveal the null structure in the key bilinear terms 
and write the system \eqref{YMF2}--\eqref{YM4} in terms of the null forms.  
In Section \ref{Sec-Reduc}, we shall rewrite this new system as a first order system and 
reduce Theorem \ref{MainThm} to proving nonlinear estimates. 
In the rest of the Sections we prove the nonlinear estimates.

\section{Null structure in the bilinear terms}
For $\mathfrak g$-valued $u,v$, define a commutator version of null forms by 
\begin{equation}\label{CommutatorNullforms}
\left\{
\begin{aligned}
  Q_0[u,v] &= [\partial_\alpha u, \partial^\alpha v] = Q_0(u,v) - Q_0(v,u),
  \\
  Q_{\alpha\beta}[u,v] &= [\partial_\alpha u, \partial_\beta v] - [\partial_\beta u, \partial_\alpha v] = Q_{\alpha\beta}(u,v) + Q_{\alpha\beta}(v,u).
\end{aligned}
\right.
\end{equation}

 Note the identity
\begin{equation}\label{NullformTrick}
  [\partial_\alpha u, \partial_\beta u]
  = \frac12 \left( [\partial_\alpha u, \partial_\beta u] - [\partial_\beta u, \partial_\alpha u] \right)
  = \frac12 Q_{\alpha\beta}[u,u].
\end{equation}

Define 
\begin{equation}\label{NewNull} 
  \mathfrak Q[u,v] = - \frac12 \varepsilon^{ijk}\varepsilon_{klm} Q_{ij}\left[R^l u^m, v \right]
  - Q_{0i}\left[R^i u_0, v \right],
\end{equation}
where $\varepsilon_{ijk}$ is the antisymmetric symbol with $\varepsilon_{123} = 1$ and
$$R_i = \angles{\nabla}^{-1}\partial_i = (1-\Delta)^{-1/2}\partial_i$$ are the Riesz transforms.

We split the spatial part $\mathbf A=(A_1,A_2, A_3)$ of the potential into divergence-free and curl-free parts and a smoother part:
\begin{equation}\label{SplitA}  
\mathbf A = \mathbf A^{\text{df}} + \mathbf A^{\text{cf}} + \angles{\nabla}^{-2} \mathbf A,
\end{equation}
where
\begin{align*}
  \mathbf A^{\text{df}}&=  \angles{\nabla}^{-2} \nabla \times \nabla \times \mathbf A,
  \\
  \mathbf A^{\text{cf}}&= -\angles{\nabla}^{-2} \nabla (\nabla \cdot \mathbf A).
\end{align*}

\subsection{ Terms of the form $[A^\alpha,\partial_\alpha \phi]$ and $ [\partial_tA^\alpha, \partial_\alpha\phi]$ }

In the Lorenz gauge, terms of the form  $[A^\alpha,\partial_\alpha \phi] $, where $A_\alpha,\phi \in \mathcal S$
with values in $\mathfrak g$, can be shown to be as a sum of bilinear null forms 
and a smoother bilinear part whereas the term $ [\partial_tA^\alpha, \partial_\alpha\phi]$ is a null form.  
\begin{lemma}\label{Lemma-Null0} In the Lorenz gauge, we have 
the identities
\begin{align}
\label{Null0}
   [A^\alpha, \partial_\alpha \phi ] & =
  \mathfrak Q\left[\angles{\nabla}^{-1} A,\phi\right] + [\angles{\nabla}^{-2}  A^\alpha, \partial_\alpha \phi ],
  \\
  \label{Null1}
  [\partial_tA^\alpha, \partial_\alpha\phi]&=  Q_{0i}\left[A^i , \phi \right].
\end{align}

\end{lemma}
\begin{proof}
To show \eqref{Null0} we modify the proof in \cite[Lemma 1 ]{st2013},
whereas \eqref{Null1} is proved in the same paper
(see identity (2.7) therein). 

Using \eqref{SplitA} we write
\begin{align*}
  A^\alpha \partial_\alpha \phi
  &=
  \left( - A_0 \partial_t \phi
  + \mathbf A^{\text{cf}} \cdot \nabla \phi \right)
  + \mathbf A^{\text{df}} \cdot \nabla \phi +  \angles{\nabla}^{-2} \mathbf A  \cdot \nabla \phi
 \end{align*}
Let us first consider the first term in the parentheses. 
We use the Lorenz gauge, $\partial_t A_0=\nabla \cdot \mathbf A  $,  to write
\begin{align*}
  \mathbf A^{\text{cf}} \cdot \nabla \phi&
  =-\angles{\nabla}^{-2} \partial^i(\partial_t A_0) \partial_i \phi=
  - \partial_t ( \angles{\nabla}^{-1} R^i A_0)\partial_i \phi.
\end{align*}
We can also write
\begin{align*}
 A_0 \partial_t \phi 
 &=-\angles{\nabla}^{-2}  \partial_i\partial^i A_0 \partial_t \phi+ \angles{\nabla}^{-2}  A_0 \partial_t \phi
  \\
  &=-\partial_i(\angles{\nabla}^{-1} R^i A_0) \partial_t \phi
  +  \angles{\nabla}^{-2}  A_0 \partial_t \phi.
\end{align*}
Combining the above identities, we get
\begin{align*}
 -A_0 \partial_t \phi +  \mathbf A^{\text{cf}} \cdot \nabla \phi
 &=Q_{i0}(\angles{\nabla}^{-1} R^i A_0, \phi)- \angles{\nabla}^{-2}  A_0 \partial_t \phi.
 \end{align*}

Next, we consider the second term. Since $$
(\mathbf A^{\text{df}})^i = \varepsilon^{ijk} \varepsilon_{klm} R_j  R^l A^m ,
$$
we have 
\begin{align*}
  \mathbf A^{\text{df}} \cdot \nabla \phi
  &=
  \varepsilon^{ijk} \varepsilon_{klm} (R_j  R^l A^m) \partial_i\phi
  \\
  &= \varepsilon^{ijk} \varepsilon_{klm} \partial_j ( \angles{\nabla}^{-1}   R^l A^m) ( \partial_i\phi)
  \\
  &=
  - \frac12 \varepsilon^{ijk}\varepsilon_{klm} Q_{ij}\left(\angles{\nabla}^{-1} R^l A_m, \phi\right).
\end{align*}

Thus, we have shown 
 \begin{equation}\label{Null2}
  \begin{split}
  A^\alpha \partial_\alpha \phi
  &=
  - \frac12 \varepsilon^{ijk}\varepsilon_{klm} Q_{ij}\left(\angles{\nabla}^{-1} R^l A_m, \phi\right)
  \\
  & \qquad + Q_{i0}(\angles{\nabla}^{-1} R^i A_0, \phi)+\angles{\nabla}^{-2}  A^\alpha \partial_\alpha \phi.
   \end{split}
  \end{equation} 
   Similarly, modifying the above argument one can show 
  \begin{equation}\label{Null3}
  \begin{split}
    \partial_\alpha \phi A^\alpha
    &=
   \frac12 \varepsilon^{ijk}\varepsilon_{klm} Q_{ij}\left( \phi, \angles{\nabla}^{-1} R^l A_m\right)
  \\
  & \qquad - Q_{i0}( \phi, \angles{\nabla}^{-1} R^i A_0)+ \partial_\alpha \phi \angles{\nabla}^{-2}  A^\alpha.
  \end{split}
  \end{equation}   
  Subtracting \eqref{Null2} and \eqref{Null3} yields \eqref{Null0}.
   \end{proof}

\subsection{ Terms of the form
$[A^\alpha,\partial_\beta A_\alpha]$ } In the Lorenz gauge, this term can be 
written as a sum of bilinear null form terms, bilinear terms which are smoother,  
a bilinear term which contains only $F$ and higher order terms in $(A, F)$.

\begin{lemma}\label{Lemma-null1} In the Lorenz gauge, we have 
the identity
\begin{align*}\label{Null4}
  [A^\alpha,\partial_\beta A_\alpha]&
=\sum_{i=1}^4\Gamma^i_\beta(A, \partial A, F, \partial F),
\end{align*}
\end{lemma}
where
\begin{equation}\label{Gammas}
 \left\{
\begin{aligned}
\Gamma^1_\beta(A, \partial A,  F, \partial F)
&=-[A_0,\partial_\beta A_0] + 
[\angles{\nabla}^{-1} R_j (\partial_t A_0),  \angles{\nabla}^{-1} R^j  \partial_t (\partial_\beta A_0)] ,
\\
\Gamma^2_\beta(A, \partial A,  F, \partial F)
&= -\frac12 \varepsilon^{ijk} 
\varepsilon_{klm} \Big\{ Q_{ij}[\angles{\nabla}^{-1}  R^n \mathbf A_n ,\angles{\nabla}^{-1} R^l \partial_\beta\mathbf A^m ] \\
&\hspace{6em}+ 
Q_{ij}[\angles{\nabla}^{-1}  R^n \partial_\beta \mathbf A_n , \angles{\nabla}^{-1}  R^l \mathbf A^m ]
 \Big\},
\\
\Gamma^3_\beta(A, \partial A,  F, \partial F)&=[ \angles{\nabla}^{-2} \nabla \times \mathbf F,
\angles{\nabla}^{-2} \nabla \times \mathbf \partial_\beta \mathbf F]
\\
&\quad
-[\angles{\nabla}^{-2}\nabla \times \mathbf F , 
\angles{\nabla}^{-2} \partial_\beta \nabla \times( \mathbf A \times \mathbf A)]
\\
& \quad- [ \angles{\nabla}^{-2} \nabla \times( \mathbf A \times \mathbf A),
\angles{\nabla}^{-2}\nabla \times \mathbf \partial_\beta \mathbf F ] 
\\
&\quad+ [\angles{\nabla}^{-2}\nabla \times( \mathbf A \times \mathbf A),
\angles{\nabla}^{-2}\partial_\beta \nabla \times( \mathbf A \times \mathbf A) ],
\\
\Gamma^4_\beta(A, \partial A,  F, \partial F)&=[\mathbf A^{\text{cf}} + \mathbf A^{\text{df}}, \angles{\nabla}^{-2} \partial_\beta \mathbf A] 
+ [\angles{\nabla}^{-2}\mathbf A,  \partial_\beta \mathbf A  ].
 \end{aligned}
\right.
\end{equation}
Here $
\mathbf F=(F_{23}, F_{31}, F_{12}).
$

Thus, $\Gamma^2_\beta$ is a combination of the commutator version $Q$-type null forms.
The term $\Gamma^1_\beta$ is also a null form (of non $Q$-type) as can be shown as follows.

We write
\begin{align*}
\Gamma^1_\beta
&=\left(-A_0 \partial_\beta A_0 + 
\angles{ \nabla}^{-1} R_j (\partial_t A_0) \angles{\nabla}^{-1} R^j  \partial_t( \partial_\beta A_0)\right)
\\
& \quad + \left(  \partial_\beta   A_0 A_0 -  \angles{\nabla}^{-1} R_j\partial_t ( \partial_\beta A_0) 
\angles{\nabla}^{-1} R_j (\partial_t A_0)  \right)
\\
&=:R_1+R_2.
\end{align*}
Now if  we denote  the space-time Fourier variables of  the first  and second $A_0$  in the product in $R_1$ by $(\tau, \xi)$  and  $(\lambda, \eta)$, respectively,  where $\tau, \ \lambda \in \R$ are temporal frequencies and  $\xi, \ \eta \in \R^3$ spatial frequencies ,  then $R_1$ has symbol 
\begin{align}\label{Nullsym}
i\left(-1+ \frac{\tau \lambda}{\angles{\xi}^2\angles{\eta}^2} \xi\cdot \eta\right)\eta_\beta
=-i\angles{\xi}^{-2}\angles{\eta}^{-2}\left(\angles{\xi}^2\angles{\eta}^2-(\tau \lambda)(\xi\cdot \eta)\right)\eta_\beta .
\end{align}
If we write 
$$
\angles{\xi}^2\angles{\eta}^2-(\tau \lambda)(\xi\cdot \eta)=\left(\angles{\xi}^2\angles{\eta}^2-|\xi|^2|\eta|^2\right)+ \left(|\xi|^2|\eta|^2-\tau \lambda\xi\cdot \eta\right), 
$$
then the first term in parenthesis satisfies the estimate 
$$
\angles{\xi}^2\angles{\eta}^2-|\xi|^2|\eta|^2=1+ 
|\xi|^2+|\eta|^2\le \angles{\xi}^2+\angles{\eta}^2.
$$
Combined with \eqref{Nullsym}, this will imply a gain in two derivatives. 
On the other hand, the second term $\left(|\xi|^2|\eta|^2-(\tau \lambda)(\xi\cdot \eta)\right)$ vanishes if $(\tau, \xi)$ and $(\lambda, \eta)$ are parallel null vectors, i.e., if  $\tau=\pm |\xi|$ and $(\lambda, \eta)=c(\tau, \xi)$ for 
some $c\in \R$.   Thus, $R_1$ is a null form up to a smooth bilinear term.  A similar argument shows 
$R_2$ is also a null form up to a smooth bilinear term.

\begin{proof}[Proof of Lemma \ref{Lemma-null1} ]
Using \eqref{SplitA}, we write 
\begin{align*}
[A^\alpha,\partial_\beta A_\alpha]&=-[A_0, \partial_\beta A_0]+ [\mathbf A, \partial_\beta \mathbf A]
\\
&=:\sum_{i=1}^4\Gamma^i_\beta(A, \partial A, F, \partial F),
\end{align*}
where
\begin{align*}
\Gamma^1_\beta(A, \partial A, F, \partial F)
&=-[A_0, \partial_\beta A_0] + [\mathbf A^{\text{cf}},  \partial_\beta \mathbf A^{\text{cf}} ],
\\
\Gamma^2_\beta(A, \partial A, F, \partial F)
&= [\mathbf A^{\text{cf}}, \partial_\beta \mathbf A^{\text{df}}] + 
[\mathbf A^{\text{df}}, \partial_\beta \mathbf A^{\text{cf}}] ,
\\
\Gamma^3_\beta (A, \partial A, F, \partial F)
&= [\mathbf A^{\text{df}}, \partial_\beta \mathbf A^{\text{df}}] ,
\\
\Gamma^4_\beta(A, \partial A, F, \partial F)
&= [\mathbf A^{\text{cf}} + \mathbf A^{\text{df}}, \angles{\nabla}^{-2} \partial_\beta \mathbf A] 
+ [\angles{\nabla}^{-2}\mathbf A,  \partial_\beta \mathbf A  ].
\end{align*}
So $\Gamma^4_\beta$ is exactly as in \eqref{Gammas}. It remains to show that $\Gamma^i_\beta$ ($i=1,2,3$) 
can also be rewritten as in \eqref{Gammas}.

First consider  $\Gamma^1_\beta$. We have 
$$
\Gamma^1_\beta=(-A_0\partial_\beta A_0 + \mathbf A^{\text{cf}} \cdot \partial_\beta \mathbf A^{\text{cf}}
 ) + (\partial_\beta A_0 A_0 - \partial_\beta\mathbf A^{\text{cf}} \cdot \mathbf A^{\text{cf}})
$$
Using the Lorenz gauge, we write 
\begin{align*}
-A_0\partial_\beta A_0 + \mathbf A^{\text{cf}} \cdot \partial_\beta \mathbf A^{\text{cf}} 
&=-A_0\partial_\beta A_0 + 
\angles{\nabla}^{-2} \nabla(\nabla \cdot \mathbf A) \cdot \partial_\beta\angles{\nabla}^{-2}\nabla(\nabla \cdot \mathbf A) \\
&=-A_0\partial_\beta A_0 + \angles{\nabla}^{-2} \nabla(\partial_t A_0) \cdot \angles{\nabla}^{-2} \nabla\partial_t (\partial_\beta A_0) \\
&=-A_0(\partial_\beta A_0)+ 
\angles{\nabla}^{-1} R_j (\partial_t A_0)  \angles{\nabla}^{-1} R^j  \partial_t (\partial_\beta A_0).
\end{align*}
Similarly, one can write
\begin{align*}
\partial_\beta A_0 A_0 - \partial_\beta\mathbf A^{\text{cf}} \cdot \mathbf A^{\text{cf}} 
&=(\partial_\beta A_0) A_0-
  \angles{\nabla}^{-1} R^j  \partial_t (\partial_\beta A_0)\angles{\nabla}^{-1} R_j (\partial_t A_0).
\end{align*}
Summing up, we obtain the desired identity for  $\Gamma^1_\beta$.

Next, we consider  $\Gamma^2_\beta$. 
 Let 
$$
B=-\angles{\nabla}^{-2}(\nabla \cdot \mathbf A)=-\angles{\nabla}^{-1} R^n \mathbf A_n, \quad 
\mathbf{C}= \angles{\nabla}^{-2}\nabla \times \mathbf A.
$$
Then
\begin{align*}
 \mathbf A^{\text{cf}}\cdot \partial_\beta \mathbf A^{\text{df}}
&=\nabla B \cdot ( \nabla \times \partial_\beta \mathbf{C}) 
= (\nabla B\times  \nabla \partial_\beta \mathbf{C}_k)^k
\\
&= \frac12  \varepsilon^{ijk} Q_{ij}( B, \partial_\beta \mathbf{C}_k)
\\
&=-\frac12 \varepsilon^{ijk} 
\varepsilon_{klm} Q_{ij}(\angles{\nabla}^{-1}R^n \mathbf A_n , \angles{\nabla}^{-1} R^l \partial_\beta\mathbf A^m ),
\end{align*}
where we used 
$  \mathbf{C}_k= \varepsilon_{klm} \angles{\nabla}^{-1} R^l \mathbf A^m .$
Similarly, 
\begin{align*}
 \mathbf A^{\text{df}}\cdot \partial_\beta \mathbf A^{\text{cf}}
 &= \frac12  \varepsilon^{ijk} Q_{ij}(\partial_\beta \mathbf{C}_k, B)
 \\
&=-\frac12 \varepsilon^{ijk} 
\varepsilon_{klm} Q_{ij}( \angles{\nabla}^{-1}  R^l \partial_\beta\mathbf A^m, \angles{\nabla}^{-1} R^n \mathbf A_n ).
\end{align*}
The terms $\partial_\beta  \mathbf A^{\text{cf}}\cdot \mathbf A^{\text{df}}$ and
$\partial_\beta  \mathbf A^{\text{df}}\cdot \mathbf A^{\text{cf}}$ in $\Gamma_\beta^2$ 
can also be written as 
in $ \mathbf A^{\text{cf}}\cdot \partial_\beta \mathbf A^{\text{df}}$ and 
$ \mathbf A^{\text{df}}\cdot \partial_\beta \mathbf A^{\text{cf}}$, respectively,
except that now $\partial_\beta$ falls on $\mathbf A_n $ instead of $\mathbf A_m $. 
Combining these facts  gives the desired identity for  $\Gamma^2_\beta$.

Finally, we consider $\Gamma^3_\beta$. Using the  definition for $F_{ij}$, we
have $\nabla \times \mathbf A=\mathbf F -\mathbf A \times \mathbf A $, where  $
\mathbf F=(F_{23}, F_{31}, F_{12}).
$ This in turn implies
$$
\mathbf A^{\text{df}}= \angles{\nabla}^{-2} \{ \nabla \times \mathbf F - \nabla \times( \mathbf A \times \mathbf A )\}.
$$
Inserting this in place of $\mathbf A^{\text{df}}$ gives the desired identity for $\Gamma^3_\beta$.

\end{proof}

\subsection{The system \eqref{YMF2}--\eqref{YM4} in terms of the null forms}

Let us now look at the nonlinear terms in  \eqref{YMF2}--\eqref{YM4}.
In view of Lemma \ref{Lemma-Null0} the first, second and third bilinear terms in \eqref{YMF2} 
are null forms up to some smoother bilinear terms. 
By the identity \eqref{NullformTrick}, the fourth and fifth
terms are identical to $2Q_0[A_\beta,A_\gamma]$ and $Q_{\beta\gamma}[A^\alpha,A_\alpha]$, respectively.
   
By Lemma \ref{Lemma-Null0}, the first term in \eqref{YM4} is a null form up to some smoother bilinear terms. 
By Lemma \ref{Lemma-null1} the second term in \eqref{YM4} is a sum of bilinear null form terms, bilinear terms 
which are smoother,  
a bilinear term which contains only $F$ and higher order terms in $(A, F)$. 

In conclusion, the system \eqref{YMF2}--\eqref{YM4} can be written as
\begin{equation}\label{AF}
\left\{
\begin{aligned}
  \square A_\beta &=  \mathfrak M_\beta(A,\partial_t A,F,\partial_t F),
  \\
  \square F_{\beta\gamma} &=  \mathfrak N_{\beta\gamma}(A,\partial_t A,F,\partial_t F),
\end{aligned}
\right.
\end{equation}
where
\begin{gather*}
   \left\{
  \begin{aligned}
  \mathfrak M_\beta(A,\partial_t A,F,\partial_t F) &= -2 \mathfrak Q[\angles{\nabla}^{-1} A,A_\beta] +
  \sum_{i=1}^4\Gamma^i_\beta(A, \partial A, F, \partial F)-2[\angles{\nabla}^{-2}  A^\alpha, \partial_\alpha A_\beta ]
  \\
 &\quad  - [A^\alpha, [A_\alpha, A_\beta]],
   \end{aligned}
  \right.
  \\
  \left\{
  \begin{aligned}
  \mathfrak N_{ij}(A,\partial_t A,F,\partial_t F)
  = &- 2\mathfrak Q[\angles{\nabla}^{-1} A,F_{ij}]
  + 2\mathfrak Q[\angles{\nabla}^{-1} \partial_j A, A_i]- 2\mathfrak Q[\angles{\nabla}^{-1} \partial_i A, A_j] 
  \\
  & + 2Q_0[A_i , A_j]
  + Q_{ij}[A^\alpha,A_\alpha]-2[\angles{\nabla}^{-2}  A^\alpha, \partial_\alpha F_{ij} ]
  \\
  &+2[\angles{\nabla}^{-2}  \partial_jA^\alpha, \partial_\alpha A_{i} ]-2[\angles{\nabla}^{-2}  \partial_i A^\alpha, \partial_\alpha A_{j} ]
  \\
  & - [A^\alpha,[A_\alpha,F_{ij}]] + 2[F_{\alpha i},[A^\alpha,A_j]] - 2[F_{\alpha j},[A^\alpha,A_i]]
  \\
  & - 2[[A^\alpha,A_i],[A_\alpha,A_j]],
  \end{aligned}
  \right.
  \\
  \left\{
  \begin{aligned}
  \mathfrak N_{0i}(A,\partial_t A,F,\partial_t F)
  = &- 2\mathfrak Q[\angles{\nabla}^{-1} A,F_{0i}]
  + 2\mathfrak Q[\angles{\nabla}^{-1} \partial_i A, A_0]- 2 Q_{0j}[A^j,A_i] + 2Q_0[A_0 , A_i]
  \\
  &
  + Q_{0i}[A^\alpha,A_\alpha]-2[\angles{\nabla}^{-2}  A^\alpha, \partial_\alpha F_{0i} ]+2[\angles{\nabla}^{-2}  \partial_i A^\alpha, \partial_\alpha A_{0} ]
  \\
  & - [A^\alpha,[A_\alpha,F_{0i}]] + 2[F_{\alpha 0},[A^\alpha,A_i]] - 2[F_{\alpha i},[A^\alpha,A_0]]
  \\
  & - 2[[A^\alpha,A_0],[A_\alpha,A_i]]
  \end{aligned}
  \right.
\end{gather*}
and  $\Gamma^i_\beta(A, \partial A, F, \partial F)$ for $i=1, \cdots, 4$ are as in \eqref{Gammas}.

\section{Reduction of Theorem \ref{MainThm} to nonlinear estimates}\label{Sec-Reduc}

\subsection{Rewriting \eqref{AF} as a first order system}
We rewrite \eqref{AF} as a first order system and reduce Theorem \ref{MainThm} 
to proving nonlinear estimates in $X^{s,b}$--spaces. 
We subtract
$A$ and $F$ to each side of the equations in \eqref{AF} and 
hence replace 
the wave operator $\square$ by the Klein-Gordon operator $\square -1$ at
the expense of adding linear term on the right hand side. 
This is done to avoid the singular operator $\abs{\nabla}^{-1} $, and instead obtain the non-singular operator
 $\angles{\nabla}^{-1} $ in the change of variables
  $$(A,\partial_t A, F, \partial_t F) \to (A_+,A_-, F_+, F_-),$$  where
\begin{align*} 
 A_\pm & = \frac12 \left( A \pm \frac{1}{i\angles{\nabla}} \partial_t A \right),
 \quad 
 F_\pm  = \frac12 \left( F \pm \frac{1}{i\angles{\nabla}} \partial_t F \right).
\end{align*}
Equivalently,
\begin{equation}\label{Substitution}
 \left\{
  \begin{aligned}
  (A, \  \partial_t A)=
  \bigl( A_+ + A_-, \   i\angles{\nabla} (A_+ - A_-))
  \\
   (F, \  \partial_t F)=
  \bigl( F_+ + F_-, \   i\angles{\nabla} (F_+ - F_-)).
  \end{aligned}
  \right.
    \end{equation}
Then the system \eqref{AF} transforms to
 \begin{equation}\label{AF1}
\left\{
\begin{aligned}
  (i\partial_t \pm \angles{\nabla}) A_{ \pm} &= \mp\frac{1}{2 \angles{\nabla}} \mathfrak{M}'(A_+, A_-, F_+, F_-),
  \\
  (i\partial_t \pm \angles{\nabla}) F_\pm &= \mp\frac{1}{2 \angles{\nabla}} \mathfrak { N}' (A_+, A_-, F_+, F_-),
 \end{aligned}
\right.
\end{equation}
where 
\begin{equation}\label{NL1}
\left\{
\begin{aligned}
  \mathfrak{M}'(A_+, A_-, F_+, F_-) &=-A + \mathfrak M_{\beta}(A,\partial_t A,F,\partial_t F),
  \\
  \mathfrak { N}' (A_+, A_-, F_+, F_-)&= -F+\mathfrak N_{\beta\gamma}(A,\partial_t A,F,\partial_t F),
  \end{aligned}
\right.
\end{equation}
In the right-hand side of \eqref{NL1} it is understood that we use the substitution \eqref{Substitution} 
on $(A, F)$ and the arguments of $\mathfrak M$ and $\mathfrak N$. 

The initial data transforms to
\begin{equation}\label{AF1data}
\left\{
\begin{aligned}
  A_{ \pm}(0) &= a_{\pm} := \frac12 \left( a \pm \frac{1}{i\angles{\nabla}} \dot a \right) \in H^{s},
  \\
   F_\pm(0) &= f_\pm  :=
   \frac12 \left( f \pm \frac{1}{i\angles{\nabla}}  \dot f \right) \in H^r .
\end{aligned}
\right.
\end{equation}

\subsection{Spaces used: $X^{s,b}$-Spaces and their properties}
We prove local well-posedness of \eqref{AF1}--\eqref{AF1data} by iterating in the $X^{s,b}$-spaces adapted to the dispersive operators $i\partial_t \pm \angles{\nabla}$. These spaces are defined to be the completion of $\mathcal S(\R^{1+3})$ with respect to the norm
$$
  \norm{u}_{X^{s,b}_\pm} = \norm{\angles{\xi}^s \bigangles{-\tau \pm \angles{\xi}}^b \widetilde u(\tau,\xi)}_{L^2_{\tau,\xi}},
$$
where $\widetilde u(\tau,\xi) = \mathcal F_{t,x} u(\tau,\xi)$ is the space-time Fourier transform of $u(t,x)$.

 Let $X^{s,b}_\pm(S_T)$ denote the restriction space to a time slab $S_T = (-T,T) \times \R^3$ for $T>0$. We recall the fact that
 \begin{equation}\label{XC-Emb}
 X^{s,b}_\pm(S_T) \hookrightarrow C([-T,T];H^s) \quad \text{for} \ b > \frac12.
\end{equation}  
 Moreover, it is well known that the linear initial value problem
 $$
  (i\partial_t \pm \angles{\nabla}) u = G  \in X^{s,b-1+\varepsilon}_\pm(S_T), \qquad u(0) = u_0\in H^s,
$$
for any $s\in \R, \ b > \frac12$,  $\ 0<\varepsilon \ll 1$,  
has a unique solution
satisfying
 \begin{equation}\label{LinearEst}
\norm{u}_{X^{s,b}_\pm(S_T)} \le C \left( \norm{u_0}_{H^s} + T^{\varepsilon} \norm{G}_{X^{s,b-1+\varepsilon}_\pm(S_T)} \right)
\end{equation}  
for $0<T<1$.

In addition to $X^{s,b}_\pm$, we shall also need the wave-Sobolev spaces $H^{s,b}$, defined to be the completion of $\mathcal S(\R^{1+3})$ with respect to the norm
$$
  \norm{u}_{H^{s,b}} = \norm{\angles{\xi}^s \bigangles{\abs{\tau}-  \angles{\xi}}^b \widetilde u(\tau,\xi)}_{L^2_{\tau,\xi}}.
$$
We shall make a frequent use of the relations
\begin{equation}\label{HXH}
\left\{
\begin{alignedat}{2}
  \norm{u}_{H^{s,b}} &\le \norm{u}_{X^{s,b}_\pm}& \quad &\text{if $b \ge 0$},
  \\
  \norm{u}_{X^{s,b}_\pm} &\le \norm{u}_{H^{s,b}}& \quad &\text{if $b \le 0$}.
\end{alignedat}
\right.
\end{equation}
In particular, \eqref{HXH} allows  us to pass from estimates in $X^{s,b}_\pm$ to corresponding estimates in $H^{s,b}$.

\subsection{Reduction to nonlinear estimates using iteration}
Using \eqref{LinearEst} and a standard iteration argument, local well-posedness can be deduced from the following nonlinear 
estimates.

\begin{lemma}\label{Lemma-NlEs} Let $0<T<1$, $(a, b)=(\frac34+\delta, \frac12+\delta)$, 
$(a', b')=(a+\delta, b+\delta)$ and $(s, r)= (1-\varepsilon, 0) $ or 
$\left(\frac67+ \varepsilon, -\frac1 {14}+ \varepsilon \right) $ where $\varepsilon$ is sufficiently small
and $0<\delta\ll \varepsilon \ll 1$. Then we have the estimates
 \begin{equation}\label{Red1}
\left\{
\begin{aligned}
  \norm{   \mathfrak{M}'(A_+, A_-, F_+, F_-)}_{X^{s-1,a'-1}_\pm(S_T)}
  &\lesssim N (1+N^3),
  \\
  \norm{ \mathfrak{N}'(A_+, A_-, F_+, F_-)}_{X^{r-1,b'-1}_\pm(S_T)}
  &\lesssim N (1+N^3),
 \end{aligned}
\right.
\end{equation} 
where
\begin{align*}
 N& = \sum_\pm\left( \norm{A_{\pm} }_{X^{s, a}_\pm(S_T)} 
  + \norm{F_\pm}_{X^{r,b}_\pm (S_T)}\right).
\end{align*}
Moreover, we have the difference estimates
\begin{equation}\label{Red2}
\left\{
\begin{aligned}
  \norm{   \mathfrak{M}'(A_+, A_-, F_+, F_-)-\mathfrak{M}'(A'_+, A'_-, F'_+, F'_-)}_{X^{s-1,a'-1}_\pm(S_T)}
  &\lesssim  \sigma(1+N^3),
  \\
  \norm{ \mathfrak{N}'(A_+, A_-, F_+, F_-)-\mathfrak{N}'(A'_+, A'_-, F'_+, F'_-)}_{X^{r-1,b'-1}_\pm(S_T)}
  &\lesssim  \sigma(1+N^3),
 \end{aligned}
\right.
\end{equation} 
where
\begin{align*}
 \sigma& = \sum_\pm\left( \norm{A_{\pm}-A'_{\pm} }_{X^{s, a}_\pm(S_T)}
  + \norm{F_\pm-F'_{\pm}}_{X^{r,b}_\pm(S_T)}\right).
\end{align*}
\end{lemma}
Then by iteration we obtain a solution $(A_+, A_-, F_+, F_-)$ of 
the transformed system \eqref{AF1}--\eqref{AF1data} on $S_T$ for $T>0$. The solution has the regularity
\begin{equation}
 \label{ReglX-soln}
 A_\pm \in X_\pm^{s, a}(S_T), \quad F_\pm \in X_\pm^{r, b}(S_T),
\end{equation}
and is unique in this space. By \eqref{XC-Emb} 
\begin{equation}
 \label{ReglC-soln}
 A_\pm \in C([-T,T]; H^{s}), \quad 
  F_\pm \in C([-T,T];  H^{r}).
\end{equation}
Continuous dependence of solution on the data and persistence of higher regularity data also follow from standard arguments
which we omit here. Once we have obtained the solution $(A_+, A_-, F_+, F_-)$ of the system \eqref{AF1},
we can then define $A=A_+ + A_-$ and $F=F_++ F_-$, and show that $(A, F)$ solves the original system \eqref{AF} (see \cite{st2013} for details).

Note that uniqueness of solution is only known only in the iteration space
\eqref{ReglX-soln} and not in  \eqref{ReglC-soln}.
But we can show that a solution
in the space \eqref{ReglC-soln} with $(s, r)=(1-\varepsilon, 0)$ also 
belongs to \eqref{ReglX-soln} with $(s, r)=\left(\frac67+ \varepsilon, -\frac1 {14}+ \varepsilon \right)  $ in which 
the solution is known to be unique. So unconditional uniqueness of the energy class solution follows from the following Lemma.
\begin{lemma}\label{Lemma-Unq}
Let $$ A_\pm \in C([-T,T]; H^{1-\varepsilon}), \quad 
  F_\pm \in C([-T,T];  L^2)
  $$
 be the solution to the system \eqref{AF1} 
 with initial data $(a_\pm, f_\pm)\in H^{1-\varepsilon}\times L^2$. Then 
 $$
 A_\pm \in X_\pm^{\frac67+ \varepsilon, \frac34 +\delta}(S_T), 
 \quad F_\pm \in X_\pm^{-\frac1{14}+\varepsilon, \frac12 +\delta}(S_T)
 $$
 where $\varepsilon$ and $\delta$ are as in Lemma \ref{Lemma-NlEs}.
\end{lemma}
The proof of Lemma \ref{Lemma-Unq} is given in the last Section.

\subsection{Further reduction of Lemma \ref{Lemma-NlEs} to null form and multilinear estimates} 
For the proof of Lemma \ref{Lemma-NlEs}, it suffices to show only \eqref{Red1} since the proof for \eqref{Red2} is similar. 

 The estimates for the linear terms in $ \mathfrak{M}' $ and $  \mathfrak{N}' $ are trivial, 
 and so we ignore them. Thus, we remain to prove the estimates for $ \mathfrak{M}$ and $ \mathfrak{N}$. So we reduce to 
\begin{equation}\label{Red11}
\left\{
\begin{aligned}
  \norm{   \mathfrak{M}(A_+, A_-, F_+, F_-)}_{X^{s-1,a'-1}_\pm}
  &\lesssim N (1+N^3),
  \\
  \norm{ \mathfrak{N}(A_+, A_-, F_+, F_-)}_{X^{r-1,b'-1}_\pm}
  &\lesssim N (1+N^3),
 \end{aligned}
\right.
\end{equation} 
where
\begin{align*}
 N& =
  \sum_\pm\left( \norm{A_{\pm} }_{X^{s, a}_\pm }
  + \norm{F_\pm}_{X^{r,b}_\pm }\right).
\end{align*}
To this end, to simplify notation, we write
 $$
 \norm{u}_{X^{s,b}} =\norm{u_+}_{X_+^{s,b}} +\norm{u_-}_{X_-^{s,b}} 
 $$
  for $u=u_++u_-$.
  
Now, looking at the terms in $ \mathfrak{M}$ and $ \mathfrak{N}$ and noting the fact that the Riesz transforms 
$R_i$ are bounded in the spaces involved, the estimates in \eqref{Red11} 
reduce to proving 
\begin{enumerate}[(i)]
 \item the  corresponding estimates for the null forms $Q=Q_0, Q_{0i}, Q_{ij}$:
\begin{align}
  \label{Nu1}
  \norm{Q[\angles{\nabla}^{-1} A, A]}_{H^{s-1,a'-1}}
  &\lesssim \norm{A}_{X^{s, a}} \norm{A}_{X^{s, a}},
  \\
    \label{Nu2}
  \norm{ Q_{ij}[\angles{\nabla}^{-1} A, \angles{\nabla}^{-1} \partial A]}_{H^{s-1, a'-1} }
  &\lesssim \norm{A}_{X^{s,a}}  \norm{A}_{X^{s,a}},
  \\
 \label{Nu3}
  \norm{ Q[\angles{\nabla}^{-1}A, F]}_{H^{r-1,b'-1} }
  &\lesssim \norm{A}_{X^{s,a}}  \norm{F}_{X^{r,b}},
    \\
  \label{Nu4}
  \norm{  Q[ A,   A]}_{H^{r-1, b'-1} }
  &\lesssim \norm{A}_{X^{s,a}}  \norm{A}_{X^{s, a}},
     \end{align}
    
\item the following estimates for $\Gamma^1$ (which is a non $Q$-type null form) and bilinear terms:
\begin{align}
  \label{Nu5}
  \norm{\Gamma^1( A, \partial A)}_{H^{s-1,a'-1}}
  &\lesssim\norm{A}_{X^{s,a}} \norm{A}_{X^{s,a}},
  \\
  \label{B1} 
   \norm{\Pi( A, \angles{\nabla}^{-2} \partial A  ) }_{H^{s-1,a'-1}}
     &\lesssim \norm{A}_{X^{s,a}} \norm{A}_{X^{s, a}},
     \\
     \label{B2} 
   \norm{ \Pi( \angles{\nabla}^{-2} A,  \partial A)   }_{H^{s-1,a'-1}}
     &\lesssim \norm{A}_{X^{s,a}} \norm{A}_{X^{s, a}},
     \\
  \label{B3} 
  \norm{\Pi(\angles{\nabla}^{-1} F, \angles{\nabla}^{-1} \partial F  ) }_{H^{s-1,a'-1}}
     &\lesssim \norm{F}_{X^{r,b}} \norm{F}_{X^{r, b}},
     \\
       \label{B4} 
   \norm{ \Pi( \angles{\nabla}^{-2} A,  \partial F)   }_{H^{r-1,b'-1}}
     &\lesssim \norm{A}_{X^{s,a}} \norm{F}_{X^{r, b}},
     \\
       \label{B5} 
   \norm{ \Pi( \angles{\nabla}^{-1} A,  \partial A)   }_{H^{r-1,b'-1}}
     &\lesssim \norm{A}_{X^{s,a}} \norm{A}_{X^{s, a}}
     \end{align}
and
\item the following trilinear and quadrilinear estimates:
 \begin{align}
   \label{T1}
   \norm{\Pi(\angles{\nabla}^{-1} F,\angles{\nabla}^{-1} \partial( AA) )}_{H^{s-1,a'-1}}
  &\lesssim  \norm{F}_{X^{r, b}} \norm{A}_{X^{s,a}} \norm{A}_{X^{s,a}},
  \\
   \label{T2}
   \norm{\Pi(\angles{\nabla}^{-1}\partial F, \angles{\nabla}^{-1}  ( AA) )}_{H^{s-1,a'-1}}
  &\lesssim  \norm{F}_{X^{r, b}} \norm{A}_{X^{s,a}} \norm{A}_{X^{s,a}},
  \\
  \label{Q1}
   \norm{\Pi(\angles{\nabla}^{-1}(AA), \angles{\nabla}^{-1} \partial ( AA)) }_{H^{s-1,a'-1}}
  &\lesssim  \norm{A}_{X^{s,a}} \norm{A}_{X^{s,a}} \norm{A}_{X^{s,a}} \norm{A}_{X^{s,a}},
  \\
   \label{T3} 
  \norm{\Pi(A,A,A)}_{H^{s-1,a'-1}}
  &\lesssim\norm{A}_{X^{s,a}} \norm{A}_{X^{s, a}} \norm{A}_{X^{s, a}},
  \\
   \label{T4}
  \norm{\Pi(A, A, F)}_{H^{r-1,b'-1}}
  &\lesssim \norm{A}_{X^{s,a}} 
  \norm{A}_{X^{s, a}}\norm{F}_{X^{r,b}},
  \\
   \label{Q2}
  \norm{\Pi(A,A, A, A)}_{H^{r-1, b'-1}}
  &\lesssim  \norm{A}_{X^{s,a}} \norm{A}_{X^{s,a}} \norm{A}_{X^{s,a}} \norm{A}_{X^{s,a}},
     \end{align}
where $\Pi(\cdots)$ denotes a multilinear operator in its arguments.
\end{enumerate}

\begin{remark}
 There is room for improvement in all of the estimates except 
 for \eqref{Nu4} and \eqref{B3} which seem to be sharp. 
 In fact, it will be clear from the estimates below that if not for one of these (or both) terms, 
 we could have improved our
 result. 
\end{remark}

\section{Proof of \eqref{Nu1}--\eqref{Nu4}  }
 
The matrix commutator null forms are linear combinations of the ordinary ones, 
in view of \eqref{CommutatorNullforms}. Since the matrix 
structure plays no role in the estimates under consideration, 
we reduce \eqref{Nu1}--\eqref{Nu4} to estimates to the ordinary null forms for $\mathbb C$-valued 
functions $u$ and $v$ (as in \eqref{OrdNullforms}).

Substituting
\begin{align*}
u&=u_+ + u_-, \quad \partial_t u = i\angles{\nabla}(u_+ - u_-),
\\
   v&=v_+ + v_-, \quad \partial_t v = i\angles{\nabla}(v_+ - v_-),
\end{align*}  
one obtains
\begin{align*}
  Q_0(u,v)
  &=
  \sum_{\pm,\pm'} (\pm 1)(\pm' 1)  \bigl[ \angles{D} u_{\pm} \angles{D} v_{\pm'} - (\pm D^i)u_{\pm} (\pm' D_i)v_{\pm'} \bigr],
  \\
  Q_{0i}(u,v)
  &=
  \sum_{\pm,\pm'} (\pm 1)(\pm' 1) \bigl[ - \angles{D} u_{\pm} (\pm' D_i)v_{\pm'} + (\pm D_i)u_{\pm} \angles{D}v_{\pm'} \bigr],
  \\
  Q_{ij}(u,v)
  &=
  \sum_{\pm,\pm'} (\pm 1)(\pm' 1) \bigl[ - (\pm D_i)u_{\pm} (\pm' D_j)v_{\pm'} + (\pm D_j)u_{\pm} (\pm' D_i)v_{\pm'} \bigr],
 \\
  Q_{ij}(u,\partial_tv)
  &=
  \sum_{\pm,\pm'} (\pm 1) \bigl[ - (\pm D_i)u_{\pm}  \angles{D} (\pm'D_j )v_{\pm'} +
  (\pm D_j)u_{\pm} \angles{D} (\pm' D_i)v_{\pm'} \bigr],
  \end{align*}
where
$$
  D = (D_1,D_2,D_3) = \frac{\nabla}{i}
$$
has Fourier symbol $\xi$. 
In terms of the Fourier symbols
\begin{align*}
  q_0(\xi,\eta) &= \angles{\xi}\angles{\eta} - \xi \cdot \eta,
  \\
  q_{0i}(\xi,\eta) &= - \angles{\xi} \eta_i + \xi_i \angles{\eta},
  \\
  q_{ij}(\xi,\eta) &= - \xi_i \eta_j + \xi_j \eta_i,
  \\
  q^t_{ij}(\xi,\eta) &=\angles{\eta} q_{ij}(\xi,\eta),
\end{align*}
we have
\begin{align*}
  Q_0(u,v) &= \sum_{\pm,\pm'} (\pm 1)(\pm' 1) B_{q_0(\pm \xi, \pm' \eta)}( u_{\pm}, v_{\pm'} ),
  \\
  Q_{0i}(u,v) &= \sum_{\pm,\pm'} (\pm 1)(\pm' 1) B_{q_{0i}(\pm \xi, \pm' \eta)}( u_{\pm}, v_{\pm'} ),
  \\
  Q_{ij}(u,v) &= \sum_{\pm,\pm'} (\pm 1)(\pm' 1) B_{q_{ij}(\pm \xi, \pm' \eta)}( u_{\pm},  v_{\pm'}),
  \\
   Q_{ij}(u,\partial_tv) &= \sum_{\pm,\pm'} (\pm 1)   B_{\angles{\eta} \cdot q_{ij}(\pm \xi, \pm' \eta)}( u_{\pm},  v_{\pm'}),
\end{align*}
where for a given symbol $\sigma(\xi,\eta)$ we denote by $B_{\sigma(\xi,\eta)}(\cdot,\cdot)$ the operator defined by
$$
  \mathcal F_{t,x} \left\{ B_{\sigma(\xi,\eta)}(u,v) \right\}(\tau,\xi) = \int \sigma(\xi-\eta,\eta) 
  \widetilde u(\tau-\lambda,\xi-\eta) \widetilde v(\lambda,\eta) \, d\lambda \, d\eta.
$$

The symbols appearing above satisfy the following estimates.
\begin{lemma}\label{Lmm-nullsyb} \cite{st2013} For all nonzero $\xi,\eta \in \R^3$,
\begin{align*}
  \abs{q_0(\xi,\eta)} &\lesssim \abs{\xi}\abs{\eta}\theta(\xi,\eta)^2 + \frac{1}{\min(\angles{\xi},\angles{\eta})},
  \\
  \abs{q_{0j}(\xi,\eta)} &\lesssim \abs{\xi}\abs{\eta}\theta(\xi,\eta) + \frac{\abs{\xi}}{\angles{\eta}} + \frac{\abs{\eta}}{\angles{\xi}},
  \\
  \abs{q_{ij}(\xi,\eta)} &\le \abs{\xi}\abs{\eta} \theta(\xi,\eta),
\end{align*}
\end{lemma}
where $\theta(\xi,\eta) = \arccos\left(\frac{\xi\cdot\eta}{\abs{\xi}\abs{\eta}}\right) \in [0,\pi]$
is the angle between $\xi$ and $\eta$.  It is this angle which quantifies the null structure in the bilinear terms. 
In particular, this angle satisfies the following estimate that allows us to trade in hyperbolic regularity 
and gain a corresponding 
amount of elliptic regularity.
\begin{lemma}\label{AngleEstimate} Let $\alpha,\beta,\gamma \in [0,1/2]$. Then for all pairs of signs $(\pm,\pm')$, all $\tau,\lambda \in \R$ and all nonzero $\xi,\eta \in \R^3$,
$$
  \theta(\pm\xi,\pm'\eta)
  \lesssim
  \left( \frac{\angles{\abs{\tau+\lambda}-\abs{\xi+\eta}}}{\min(\angles{\xi},\angles{\eta})} \right)^\alpha
  +
  \left( \frac{\angles{-\tau\pm\abs{\xi}}}{\min(\angles{\xi},\angles{\eta})} \right)^\beta
  +
  \left( \frac{\angles{-\lambda\pm'\abs{\eta}}}{\min(\angles{\xi},\angles{\eta})} \right)^\gamma.
$$
\end{lemma}
For a proof, see for example \cite[Lemma 2.1]{Selberg2008}.

In view of Lemma \ref{Lmm-nullsyb}, and since the norms we use only depend on 
the absolute value of the space-time Fourier transform, we can reduce any estimate for $Q(u,v)$ to
a corresponding estimate for the three expressions:
$$
  B_{\theta(\pm\xi,\pm'\eta)}(\abs{\nabla} u, \abs{\nabla} v),
  \quad
  \angles{\nabla} u \angles{\nabla}^{-1} v
  \quad \text{and} \quad
  \angles{\nabla}^{-1} u \angles{\nabla} v.
$$
Consequently, we can reduce \eqref{Nu1}--\eqref{Nu4} to the following:
\begin{equation}  \label{NLE}
\left\{
\begin{aligned}
  \norm{B_{\theta(\pm\xi,\pm'\eta)}(u, v)}_{H^{s-1,a'-1}}
  &\lesssim \norm{u}_{X^{s,a}_\pm} \norm{v}_{X^{s-1,a}_{\pm'}},
  \\
  \norm{B_{\theta(\pm\xi,\pm'\eta)}(u, v)}_{H^{r-1,b'-1}}
  &\lesssim \norm{u}_{X^{s,a}_\pm} \norm{v}_{X^{r-1,b}_{\pm'}},
  \\
  \norm{B_{\theta(\pm\xi,\pm'\eta)}(u, v)}_{H^{r-1,b'-1}}
  &\lesssim \norm{u}_{X^{s-1,a}_\pm} \norm{v}_{X^{s-1,a}_{\pm'}},
  \\
  \norm{uv}_{H^{s-1,a'-1}}
  &\lesssim \norm{u}_{H^{s, a}} \norm{v}_{H^{s+1,a}},
  \\
  \norm{uv}_{H^{s-1,a'-1}}
  &\lesssim \norm{u}_{H^{s+2,a}} \norm{v}_{H^{s-1,a}},
  \\
  \norm{uv}_{H^{r-1,b'-1}}
  &\lesssim \norm{u}_{H^{s,a}} \norm{v}_{H^{r+1, b}},
  \\
  \norm{uv}_{H^{r-1,b'-1}}
  &\lesssim \norm{u}_{H^{s+1,a}} \norm{v}_{H^{r-1, b }},
  \\
  \norm{uv}_{H^{r-1,b'-1}}
  &\lesssim \norm{u}_{H^{s-1,a}} \norm{v}_{H^{s+1,a}},
  \\
  \norm{uv}_{H^{r-1,b'-1}}
  &\lesssim \norm{u}_{H^{s+1,a}} \norm{v}_{H^{s-1,a}},
\end{aligned}
\right.
\end{equation}
where the first, fourth and fifth estimates are reductions of \eqref{Nu1}--\eqref{Nu2} and 
the rest comes from \eqref{Nu3}--\eqref{Nu4}.

The estimates in \eqref{NLE} follow from the following atlas of product estimates
in $H^{s,b}$--spaces which is due to D'Ancona, Foschi and Selberg \cite{dfs2012}
and null form estimates which is proved in \cite{st2013}.
In fact, the null form estimates follow from the estimate for the angles in Lemma \ref{AngleEstimate} and 
the product estimates in
Theorem \ref{ThmAtlas} below.
\begin{theorem}[Product estimates, \cite{dfs2012}]\label{ThmAtlas} Let $s_0,s_1,s_2 \in \R$ and $b_0,b_1,b_2 \ge 0$. Assume that
\begin{equation*}
\left\{
\begin{aligned}
  \sum b_i &> \frac12,
  \\
  \sum s_i &> 2 - \sum b_i,
  \\
  \sum s_i& > \frac32 - \min_{i \neq j}(b_i+b_j),
  \\
  \sum s_i& > \frac32 - \min(b_0 + s_1 + s_2, s_0 + b_1 + s_2, s_0 + s_1 + b_2),
  \\
  \sum s_i& \ge 1,
  \\
  \min_{i \neq j}(s_i+s_j) &\ge 0,
\end{aligned}
\right.
\end{equation*} 
and that the last two inequalities are \textbf{not both equalities}. Then
$$  
\norm{uv}_{H^{-s_0,b_0}} \lesssim \norm{u}_{H^{s_1,b_1}} \norm{v}_{H^{s_2,b_2}}
$$  
holds for all $u,v \in \mathcal S(\R^{1+3})$.
  \end{theorem}
  
 \begin{theorem}[Null form estimates, \cite{st2013} ] \label{NullformThm}  
Let $\sigma_0,\sigma_1,\sigma_2,\beta_0,\beta_1,\beta_2 \in \R$. Assume that
\begin{equation*}
\left\{
\begin{aligned} \label{NullformThm:1}
 & 0 \le \beta_0 < \frac12 < \beta_1,\beta_2 < 1,
  \\
 & \sum \sigma_i + \beta_0 > \frac32 - (\beta_0 + \sigma_1 + \sigma_2),
  \\
 & \sum \sigma_i > \frac32 - (\sigma_0 + \beta_1 + \sigma_2),
  \\
&  \sum \sigma_i > \frac32 - (\sigma_0 + \sigma_1 + \beta_2),
  \\
  &\sum \sigma_i + \beta_0 \ge 1,
  \\
 & \min(\sigma_0 + \sigma_1, \sigma_0 + \sigma_2, \beta_0 + \sigma_1 + \sigma_2) \ge 0,
\end{aligned}
\right.
\end{equation*} 
and that the last two inequalities are \textbf{not both equalities}. Then we have the null form estimate
$$
  \norm{B_{\theta(\pm\xi,\pm'\eta)}(u,v)}_{H^{-\sigma_0,-\beta_0}}
  \lesssim
  \norm{u}_{X^{\sigma_1,\beta_1}_\pm} \norm{v}_{X^{\sigma_2,\beta_2}_{\pm'}}.
$$
\end{theorem}

 \section{Proof of \eqref{Nu5}--\eqref{B5}}
 
 \subsection{Proof of \eqref{Nu5}}
We write $\Gamma^1=(\Gamma^1_0, \Gamma^1_1, \Gamma^1_2,\Gamma^1_3)$.
Using the substitution \eqref{Substitution}, we write
\begin{align*}
\Gamma^1_0(A, \partial A)
&=-A_0 \partial^i \mathbf A_i + 
\angles{ \nabla}^{-1} R_j (\partial_t A_0) \angles{\nabla}^{-1} R^j  \partial^i ( \partial_t \mathbf A_i)
\\
& \quad +  \partial^i \mathbf A_i A_0 -  \angles{\nabla}^{-1} R^j\partial^i ( \partial_t \mathbf A_i) 
\angles{\nabla}^{-1} R_j (\partial_t A_0) 
\\
&=- \sum_{\pm, \pm'} \Bigl\{ A_{0, \pm} \partial^i \mathbf A_{i, \pm'} + 
 R_j (\pm A_{0,  \pm}) 
R^j   (\pm' \partial^i  \mathbf A_{i, \pm'}  )       
\\
&\qquad \qquad -\partial^i \mathbf A_{i, \pm}A_{0, \pm'} -
R^j  (\pm \mathbf \partial^i \mathbf A_{i, \pm}  )  R_j (\pm' A_{0,  \pm'}) 
\Bigr\},
\end{align*}
where in the first two lines
 we used the Lorenz gauge 
condition $\partial_t A_0=\partial^i \mathbf A_i$ to replace $\partial_t A_0$ by $\partial^i \mathbf A_i$.
We do this in order to avoid
too much time derivative in the nonlinearity.
 
Similarly, 
\begin{align*}
\Gamma_i^1 (A, \partial A)
&=-A_0 \partial_i A_0 + 
\angles{\nabla}^{-1} R_j (\partial_t A_0) \angles{\nabla}^{-1} R^j  \partial_t (\partial_i A_0)
\\
&\qquad  + \partial_i A_0 A_0 - \angles{\nabla}^{-1} R^j  \partial_t (\partial_i A_0)
\angles{\nabla}^{-1} R_j (\partial_t A_0)
\\
&=  -\sum_{\pm, \pm'} \Bigl\{ A_{0, \pm} \partial_i A_{0, \pm'} + 
 R_j (\pm A_{0,  \pm}) 
R^j   (\pm'   \partial_i A_{0, \pm'}  )        \bigr.
\\
&\qquad \qquad -\partial_i A_{0, \pm}A_{0, \pm'} -
R^j   (\pm  \partial_iA_{0, \pm}  ) R_j (\pm' A_{0,  \pm'}) 
\Bigr\}.
\end{align*}

We can write
\begin{align}
\label{Nu5-0}
\Gamma^1_0(A, \partial A)
&= - \sum_{\pm, \pm'}\bigl\{ \mathcal P_{\pm, \pm'} ( A_{0, \pm},  \partial^i\mathbf A_{i, \pm'} )
-
 \mathcal P_{\pm, \pm'} ( \partial^i \mathbf A_{i, \pm}, A_{0, \pm'} )\bigr\},
 \\
 \label{Nu5-i}
 \Gamma_i^1 (A, \partial A)&=    -\sum_{\pm, \pm'}\bigl\{ \mathcal P_{\pm, \pm'} ( A_{0, \pm}, 
 \partial_i A_{0, \pm'} )-
 \mathcal P_{\pm, \pm'} ( \partial_i  A_{0, \pm}, A_{0, \pm'} )\bigr\},
\end{align}
where
\begin{align}
\label{Nu5-1}
\mathcal P_{\pm, \pm'} (u, v)=uv+  
R_j    (\pm u )  R^j (\pm' v)
 \end{align}
which has symbol
\begin{align}
\label{Nu5-symb}
 p_{\pm, \pm'} (\eta, \zeta)=1-\frac{ (\pm \eta )\cdot (\pm'\zeta)}{\angles{\eta} \angles{\zeta}}.
 \end{align}
This symbol satisfies the estimate
\begin{align}
\label{Nu5-symbE}
 |p_{\pm, \pm'} (\eta, \zeta)|\lesssim
  \theta^2(\pm \eta, \pm' \zeta )+ 
 \frac{ 1 }{\angles{\eta} \angles{\zeta}}.
 \end{align}
Indeed, if $\pm=+$ and $\pm'=\pm$, then
\begin{align*}
 p_{+, \pm} (\eta, \zeta) = 1-\frac{  \eta \cdot ( \pm \zeta) }{\angles{\eta} \angles{\zeta}}
 = \frac{( |\eta||\zeta|- \eta \cdot (\pm\zeta)) +
 (\angles{\eta} \angles{\zeta}-|\eta||\zeta|) }{\angles{\eta} \angles{\zeta}}.
  \end{align*}
But $$
|\eta||\zeta|- \eta \cdot (\pm \zeta)\le |\eta||\zeta|(1- \cos(\theta( \eta, \pm \zeta )))\lesssim  
|\eta||\zeta|\theta^2( \eta,  \pm\zeta )
$$
and 
$$
\angles{\eta} \angles{\zeta}-
|\eta||\zeta|=\frac{1+|\eta|^2+|\zeta|^2}{\angles{\eta} \angles{\zeta}+ |\eta||\zeta|  } \le 1.
$$
Then the estimate \eqref{Nu5-symbE} follows.

So in view of \eqref{Nu5-0}--\eqref{Nu5-symbE} and by symmetry, the estimate \eqref{Nu5} can be reduced to the following:
\begin{equation*}
\left\{
\begin{aligned} 
  \norm{B_{\theta(\pm\xi,\pm'\eta)}(u, v)}_{H^{s-1,a'-1}}
  &\lesssim \norm{u}_{X^{s,a}_\pm} \norm{v}_{X^{s-1,a}_{\pm'}},
  \\
  \label{Nu5-2}
   \norm{uv}_{H^{s-1,a'-1}}
  &\lesssim \norm{u}_{H^{s+1,a}} \norm{v}_{H^{s,a}},
      \end{aligned}
\right.
\end{equation*}
where the first estimate reduces to the first estimate in \eqref{NLE} while the second estimate 
holds by Theorem \ref{ThmAtlas}.
 
 \subsection{Proof of \eqref{B1}--\eqref{B5}}
Using \eqref{Substitution}, if necessary, the estimates \eqref{B1}--\eqref{B5} reduce to
\begin{equation*}
\left\{
\begin{aligned} 
   \norm{uv }_{H^{s-1,a'-1}}
     &\lesssim \norm{u}_{H^{s,a}} \norm{v}_{H^{s+1, a}}
     \\
   \norm{ uv  }_{H^{s-1,a'-1}}
     &\lesssim \norm{u}_{H^{s+2,a}} \norm{v}_{H^{s-1, a}}
     \\
  \norm{uv  }_{H^{s-1,a'-1}}
     &\lesssim \norm{u}_{H^{r+1,b}} \norm{v}_{H^{r, b}}
     \\
   \norm{uv     }_{H^{r-1,b'-1}}
     &\lesssim \norm{u}_{H^{s+2,a}} \norm{v}_{H^{r-1, b}}
     \\
   \norm{ uv  }_{H^{r-1,b'-1}}
     &\lesssim \norm{u}_{H^{s+1,a}} \norm{v}_{H^{s-1, a}},
     \end{aligned}
\right.
\end{equation*}
     all of which hold by Theorem \ref{ThmAtlas}.

\subsection{Proof of \eqref{T1}--\eqref{Q2} }

 The estimates \eqref{T1}--\eqref{Q2} reduce to the following:
\begin{align}
  \label{T1-1}
   \norm{u  \angles{\nabla }^{-1}( vw) }_{H^{s-1,a'-1}}
  &\lesssim \norm {u}_{H^{r+1,b}} \norm{v}_{H^{s,a}} \norm{w}_{H^{s-1,a}},
  \\
  \label{T2-1}
   \norm{u \angles{\nabla }^{-1}(vw) }_{H^{s-1,a'-1}}
  &\lesssim \norm{ u}_{H^{r,b}} \norm{v}_{H^{s,a}} \norm{w}_{H^{s,a}},
  \\
  \label{Q1-1}
   \norm{\angles{\nabla }^{-1}(uv)\angles{\nabla }^{-1}( wz) }_{H^{s-1,a'-1}}
  &\lesssim \norm{u}_{H^{s,a}} \norm{v}_{H^{s,a}} \norm{w}_{H^{s,a}} \norm{z}_{H^{s-1,a}}
  \\
  \label{T3-1}
  \norm{uvw   }_{H^{s-1,a'-1}}
  &\lesssim \norm{u}_{H^{s,a}} \norm{v}_{H^{s,a}} \norm{w}_{H^{s,a}},
  \\
  \label{T4-1}
    \norm{uvw   }_{H^{r-1,b'-1}}
  &\lesssim \norm{u}_{H^{s,a}} \norm{v}_{H^{s,a}} \norm{w}_{H^{r,b}},
  \\
  \label{Q2-1}
   \norm{uvwz   }_{H^{r-1,b'-1}}
  &\lesssim \norm{u}_{H^{s,a}} \norm{v}_{H^{s,a}} \norm{w}_{H^{s,a}} \norm{z}_{H^{s,a}}.
 \end{align}

The trilinear estimates \eqref{T1-1}, \eqref{T2-1}, \eqref{T3-1} and \eqref{T4-1} follows by applying Theorem \ref{ThmAtlas}
twice as follows:
  \begin{align*}
   \norm{u \angles{\nabla }^{-1}( vw) }_{H^{s-1,a'-1}}
  &\lesssim \norm {u}_{H^{r+1,b}} \norm{\angles{\nabla }^{-1}(vw)}_{H^{s-\frac12,0}} 
  \\
  &=\norm {u}_{H^{r+1,b}} \norm{vw}_{H^{s-\frac32,0}} 
  \\
  &\lesssim \norm {u}_{H^{r+1,b}} \norm{v}_{H^{s,a}} \norm{w}_{H^{s-1,a}},
  \\
 \norm{u \angles{\nabla }^{-1}( vw) }_{H^{s-1,a'-1}}
     &\lesssim \norm {u}_{H^{r, b}} \norm{\angles{\nabla }^{-1}(vw)}_{H^{s+\frac12,0}} 
  \\
  &=\norm {u}_{H^{r, b}} \norm{vw}_{H^{s-\frac12,0}} 
  \\
  &\lesssim \norm {u}_{H^{r,b}} \norm{v}_{H^{s,a}} \norm{w}_{H^{s,a}},
  \\ 
  \norm{uvw   }_{H^{s-1,a'-1}}
  &\lesssim \norm{u}_{H^{s,a}} \norm{vw}_{H^{s-\frac12,0}} 
  \\
 &\lesssim\norm{u}_{H^{s,a}} \norm{v}_{H^{s,a}} \norm{w}_{H^{s,a}},
 \\
 \norm{uvw   }_{H^{r-1,b'-1}}
  &\lesssim \norm{uv}_{H^{s-\frac12,0}} \norm{w}_{H^{r,b}}
  \\
 &\lesssim\norm{u}_{H^{s,a}} \norm{v}_{H^{s,a}} \norm{w}_{H^{r,b}}.
     \end{align*}
 The quadrilinear \eqref{Q1-1} and  \eqref{Q2-1} follow by applying Theorem \ref{ThmAtlas} and \cite[Theorem 8.1]{dfs2012}
follows:
      \begin{align*}
     \norm{\angles{\nabla }^{-1}(uv)\angles{\nabla }^{-1}( wz) }_{H^{s-1,a'-1}}
  &\lesssim \norm{\angles{\nabla }^{-1}(uv)}_{H^{s+\frac27,\frac12 }} \norm{
  \angles{\nabla }^{-1}( wz)}_{H^{0, 0 }} 
  \\
  &= \norm{uv}_{H^{s-\frac57,\frac12 }} 
  \norm{  wz }_{H^{-1, 0 }} 
  \\
   &\lesssim \norm{u}_{H^{s,a}} \norm{v}_{H^{s,a}} \norm{w}_{H^{s,a}} \norm{z}_{H^{s-1,a}},
   \\
   \norm{uvwz }_{H^{r-1,b'-1}}&\lesssim  \norm{uv }_{H^{\frac27, \frac14} } \norm{uvwz }_{H^{\frac27, \frac14}}
   \\
  &\lesssim \norm {u}_{H^{s, a}}   \norm {v}_{H^{s, a}} \norm {w}_{H^{s, a}} \norm {z}_{H^{s, a}}.
     \end{align*}
     
 \section{Proof of Lemma \ref{Lemma-Unq}}
 By assumption the solution $(A_\pm, F_\pm)$ to the system \eqref{AF1} lies in the regularity class
\begin{equation}\label{EnergyCl}
 A_\pm \in C([-T,T]; H^{1-\varepsilon}), \quad 
  F_\pm \in C([-T,T];  L^2).
\end{equation}
 For sufficiently small $\sigma$ such that $1\gg \sigma \gg \varepsilon$, we claim
 \begin{align}
  \label{AX}
  A_\pm \in X_\pm^{1-\sigma, 1-\sigma}(S_T),
  \\
  \label{FX}
  F_\pm \in X_\pm^{-\sigma, \frac12+\sigma}(S_T),
 \end{align}
which are in fact far better than the desired regularities in Lemma \ref{Lemma-Unq}.
 
  To prove \eqref{AX} we start with the assumed regularity for $A_\pm$ in \eqref{EnergyCl}
  and use the equations for $A_\pm$ to successively improve the regularity by
applying Strichartz, H\"{o}lder and Sobolev inequalities until we reach \eqref{AX}. 
The regularity of $F_\pm$ in \eqref{FX}
is determined from \eqref{AX} and the relation  \eqref{Curvature} by applying the product estimates in
Theorem \ref{ThmAtlas}. We do not need the null structures in the nonlinear terms, and so we write the equations 
for $A_\pm$ in the original form.

Consider the solutions $A_\pm$ to the equations
 \begin{equation}\label{Schem}
  (i\partial_t \pm \angles{\nabla}) A_{ \pm} = \mp\frac{1}{2 \angles{\nabla}} M'(A_+, A_-),
\end{equation}
where 
$$
M'(A_+, A_-)=-A - 2[A^\alpha,\partial_\alpha A] + [A^\alpha,\partial A_\alpha] - 
  [A^\alpha, [A_\alpha,A]].
$$
We have thus replaced $\mathfrak{M}'$ in \eqref{AF1} which was written in terms
of the null forms by $M'$ which is the same but in its original form.
Now write 
$$
A_\pm= A_\pm^{(0)} +  A_\pm^{(1)} +  A_\pm^{(2)},
$$
where $ A_\pm^{(0)} $
is the homogeneous part while $  A_\pm^{(1)} $ and $ A_\pm^{(2)} $ are the inhomogeneous parts
corresponding to the linear and nonlinear terms, respectively.
First note that by \eqref{LinearEst} and assumption \eqref{EnergyCl}
 $$
  A_\pm^{(0)} , \   A_\pm^{(1)} \in X_\pm^{1-\varepsilon, 1}(S_T)
 $$
 which imply (also using \eqref{EnergyCl})
 $$
  A_\pm^{(2)}=A_\pm- A_\pm^{(0)} -  A_\pm^{(1)} \in X_\pm ^{1-\varepsilon , 0}(S_T).
 $$
 
 It remains to prove
 \begin{equation}
  \label{Clm}
  A_\pm^{(2)} \in X_\pm^{1-\sigma, 1-\sigma}(S_T).
 \end{equation}

 To this end,
we start with \eqref{EnergyCl} and use th following Strichartz estimates to successively
improve the regularity. 
\begin{lemma}\label{LemmaStr}\cite[Lemma 7.1]{Selberg2010b} Suppose $2 < q \le \infty$ and $2 \le r < \infty$ satisfy $\frac12 \le \frac{1}{q} + \frac{1}{r} \le 1$. Then
$$
  \norm{u}_{L_t^q L_x^r}
  \lesssim
  \bignorm{\abs{\nabla}^{1-\frac{2}{r}}u}_{H^{0,1 - (\frac{1}{q} + \frac{1}{r}) + \gamma }}
$$
holds for any $\gamma > 0$.
\end{lemma} 
 In what follows we always assume $\gamma$ to be sufficiently small.

 We also need the following simple inequality: 
  \begin{equation}\label{Lbnz}
  \norm{f|\nabla|g}_{L^2} \lesssim 
  \norm{|\nabla|^{\varepsilon}f|\nabla|^{1-\varepsilon}g}_{L^2} +
  \norm{ f|\nabla|^{1-\varepsilon} g}_{H^\varepsilon}.
 \end{equation} 
 Indeed, let the Fourier variables of $f$ and $g$ be $\eta$ and $\xi-\eta$, 
respectively,  
so that $\xi$ is the output frequency for $f|\nabla|g$. 
Now, if $\abs{\xi}\lesssim \abs{\eta}\sim \abs{\xi-\eta}$ then
the left hand side of \eqref{Lbnz} is bounded by the first term on the right hand side, while if 
$\abs{\eta}\lesssim \abs{\xi-\eta}\sim \abs{\xi}$ or $\abs{\xi-\eta}\lesssim \abs{\eta}\sim \abs{\xi}$ it is 
bounded by the second term on the right hand side.

\subsection{First estimate for $A_\pm^{\text{(2)}}$}
 We claim that
\begin{equation}\label{A1}
  A_\pm^{\text{(2)}} \in X_\pm^{\frac23-2\varepsilon , \frac56-2\varepsilon}(S_T).
  \end{equation}
By \eqref{LinearEst} and using \eqref{Substitution}, it suffices to show 
  $$
  A|\nabla| A , \ A^3 \in X_\pm^{-\frac13-2\varepsilon , -\frac16-2\varepsilon}(S_T). 
  $$
   In view of \eqref{Lbnz}, we reduce to proving
  $$
  A|\nabla|^{1-\varepsilon}A \in  X_\pm^{-\frac13-\varepsilon , -\frac16-2\varepsilon}(S_T), \quad 
  |\nabla|^{\varepsilon} A|\nabla|^{1-\varepsilon} ,
  \ A^3 \in X_\pm^{-\frac13-2\varepsilon , -\frac16-2\varepsilon}(S_T).
  $$
  
 Using Lemma \ref{LemmaStr} and duality, we have for some $q'<2$
 \begin{align*}
  \norm{A|\nabla|^{1-\varepsilon} A}_{X_\pm^{-\frac13-\varepsilon , -\frac16-2\varepsilon} (S_T)}
  &\lesssim \norm{A|\nabla|^{1-\varepsilon} A }_{ L_t^{q'} L_x^{\frac6{4+3\varepsilon } }(S_T)}\\
  &\lesssim \norm{A }_{L_t^\infty L_x^{\frac6{1+3\varepsilon}} (S_T)} 
  \norm{|\nabla|^{1-\varepsilon} A }_{L_t^\infty L_x^2 (S_T)}
  \\
  &\lesssim \norm{A }_{L_t^\infty H^{1-\varepsilon} (S_T)}^2,
 \end{align*}
where we used Sobolev embedding in the last inequality.
Similarly,
 \begin{align*}
  \norm{|\nabla|^{\varepsilon} A|\nabla|^{1-\varepsilon} A}_{X_\pm^{-\frac13-2\varepsilon , -\frac16-2\varepsilon}}
  &\lesssim \norm{|\nabla|^{\varepsilon} A|\nabla|^{1-\varepsilon} A }_{ L_t^{q'} L_x^{\frac3{2+3\varepsilon }}(S_T)}\\
  &\lesssim \norm{|\nabla|^{\varepsilon}A }_{L_t^\infty L_x^{\frac6{1+6\varepsilon}} (S_T)}
  \norm{|\nabla|^{1-\varepsilon} A }_{L_t^\infty L_x^2(S_T)}
  \\
  &\lesssim \norm{A }_{L_t^\infty H^{1-\varepsilon}(S_T)}^2
 \end{align*}
 and
 \begin{align*}
 \norm{A^3}_{X_\pm^{-\frac13-2\varepsilon , -\frac13-2\varepsilon}(S_T)}
  &\lesssim \norm{ A^3 }_{ L_t^{q'} L_x^{\frac3{2+3\varepsilon }}(S_T)}
  \lesssim \norm{A }^3_{L_t^\infty L_x^{\frac9{2+3\varepsilon}}(S_T)} 
  \lesssim \norm{A }_{L_t^\infty H^{1-\varepsilon}(S_T)}^3.
 \end{align*}

\subsection{Inductive estimates} We claim that, for $m=1,2,\dots$,
\begin{equation}\label{Am}
  A_\pm^{(2)} \in X_\pm^{s_m-\varepsilon_m,b_m-\varepsilon_m}(S_T) ,
\end{equation}
where
$$
  s_m = 1-\frac{1}{2^m+1},
  \quad
  b_m = 1-\frac{1}{2^{m+1}+2}, \quad \varepsilon_m=(m+1)\varepsilon.
$$
Granted this claim we can then choose $m$ sufficiently large to obtain \eqref{Clm}.

Note that the claim  \eqref{Am} is true for $m=1$ by \eqref{A1}. We shall prove that if \eqref{Am} holds for some $m \ge 1$, 
then it holds for $m+1$ also.

Interpolating \eqref{Am} with $A_\pm^{\text{(2)}} \in X_\pm^{1-\varepsilon,0}(S_T)$, we get
$$
  A_\pm^{\text{(2)}} 
  \in X_\pm^{\theta(1-\varepsilon) + (1-\theta)(s_m-\varepsilon_m),(1-\theta)(b_m-\varepsilon_m)}(S_T)
$$
for all $0 \le \theta \le 1$ and $\varepsilon > 0$. Take $\theta = \frac{2^m}{2^{m+1}+1}$ to obtain
\begin{equation}\label{Am1}
  A_\pm^{\text{(2)}} \in X_\pm^{s_{m+1}-\varepsilon_m,\frac12-\varepsilon_m}(S_T).
\end{equation}
Then by Lemma \ref{LemmaStr} 
\begin{equation}\label{Am11}
  A_\pm^{\text{(2)}} \in L_t^qL_x^r(S_T), \quad  
  |\nabla|^{\varepsilon} A_\pm^{\text{(2)}} \in  L_t^{ \frac{2q}{2-q\varepsilon }}L_x^{\frac{2r}{2+r\varepsilon }
 } (S_T),
\end{equation}
where
$$
\frac1q=\frac{2^{m+1}}{2^{m+2}+2}+\frac{\varepsilon_m}2+\gamma, \quad 
\frac1r=\frac{1}{2^{m+2}+2}+\frac{\varepsilon_m}2.
$$
Recall also that $A_\pm^{\text{(0)}}$ and $A_\pm^{\text{(1)}}$ has better regularity than \eqref{Am1}.
This in turn implies they also satisfy the property \eqref{Am11}, and we can therefore conclude
\begin{equation}\label{Am12}
  A_\pm \in L_t^qL_x^r(S_T), \quad  
  |\nabla|^{\varepsilon} A_\pm \in  L_t^{ \frac{2q}{2-q\varepsilon }}L_x^{\frac{2r}{2+r\varepsilon }
 } (S_T).
\end{equation}

 Using \eqref{LinearEst}, the induction claim
$$
 A_\pm^{(2)} \in X_\pm^{s_{m+1}-\varepsilon_{m+1},b_{m+1}-\varepsilon_{m+1}}(S_T)
$$
reduces to 
$$
A|\nabla| A,  \quad A^3 \in  X_\pm^{-1+s_{m+1}-\varepsilon_{m+1},-1+b_{m+1}-\varepsilon_{m+1}}(S_T).
$$

For the quadratic term $A|\nabla| A$, in view of \eqref{Lbnz}, it suffices to show
\begin{align*}
 A|\nabla|^{1-\varepsilon} A &\in  X_\pm^{-1+s_{m+1}-\varepsilon_{m},-1+b_{m+1}-\varepsilon_{m+1}}(S_T),
 \\
|\nabla|^{\varepsilon} A|\nabla|^{1-\varepsilon} A &\in 
X_\pm^{-1+s_{m+1}-\varepsilon_{m+1},-1+b_{m+1}-\varepsilon_{m+1}}(S_T).
\end{align*}
By Lemma \ref{LemmaStr} and duality, we thus reduce to proving for some $l\le 2$
 $$
 A|\nabla|^{1-\varepsilon} A  \in L_t^lL_x^{\frac{2r}{2+r}}(S_T),
 \quad |\nabla|^{\varepsilon} A|\nabla|^{1-\varepsilon} A, \ A^3  \in L_t^lL_x^{\frac{2r}{2+r(1+\varepsilon)}}(S_T).
 $$
But by H\"{o}lder and Sobolev inequality
 \begin{align*}
 \norm{ A|\nabla|^{1-\varepsilon} A }_{ L_t^{l} L_x^{\frac{2r}{2+r }} (S_T)}
    \lesssim \norm{ A }_{ L_t^q L_x^r (S_T)} \norm{A }_{L_t^\infty H^{1-\varepsilon} (S_T)},
 \end{align*}
  \begin{align*}
 \norm{  |\nabla|^{\varepsilon}A|\nabla|^{1-\varepsilon} A }_{ L_t^{l} L_x^{\frac{2r}{2+r(1+\varepsilon) }}(S_T)}
  \lesssim \norm{ |\nabla|^{\varepsilon} A }_{L_t^{ \frac{2q}{2-q\varepsilon }}L_x^{\frac{2r}{2+r\varepsilon }}(S_T)}
  \norm{A }_{L_t^\infty H^{1-\varepsilon}(S_T)}
   \end{align*}
   and
 \begin{align*}
 \norm{A^3}_{L_t^{l} L_x^{\frac{2r}{2+r(1+\varepsilon) }} (S_T)}
   \lesssim \norm{A }^3_{L_t^\infty L_x^{\frac{6r}{2+ r(1+\varepsilon)}}(S_T)} 
  \lesssim \norm{A }_{L_t^\infty H^{1-\varepsilon}(S_T)}^3,
 \end{align*}
 where in the first two inequalities we used \eqref{Am12}.

 \subsection{Proof of \eqref{FX}}
 
Using \eqref{Curvature} and \eqref{Substitution}, we can write
\begin{equation}\label{CurvS}
\left\{
\begin{aligned} 
  F_{ij, \pm}&= \partial_i A_{j, \pm}- \partial_j A_{i, \pm}+ \frac12\left( [A_i, A_j] \pm 
  \frac{1}{i\angles{\nabla}} \partial_t  [A_i, A_j]  \right),
  \\
  F_{0i, \pm}&= \partial_t A_{i, \pm}- \partial_i A_{0, \pm}+ \frac12\left( [A_0, A_i] \pm 
  \frac{1}{i\angles{\nabla}} \partial_t  [A_0, A_i]  \right),
 \end{aligned}
\right.
\end{equation} 
 where $\partial_t A$ in the parentheses can also be written in terms of $A_\pm$.
 
 In view of \eqref{CurvS}, the claim \eqref{FX} reduces to 
\begin{align}
\label{Aq}
  &A^2, \quad  \angles{\nabla}^{-1}( A\partial A) \in X_\pm^{-\sigma, \frac12+\sigma}(S_T),
  \\
  \label{Al}
 & \partial A_\pm  \in X_\pm^{-\sigma, \frac12+\sigma}(S_T).
\end{align}

Let us first prove \eqref{Aq}. Using \eqref{AX} and the substitution \eqref{Substitution}, this reduces to
\begin{align}
\label{Aq1}
 \norm{uv}_{X_\pm^{-\sigma, \frac12+\sigma}}&\lesssim  \norm{u}_{X_{\pm_1}^{1-\sigma, 1-\sigma}}
 \norm{v}_{X_{\pm_2}^{1-\sigma, 1-\sigma}},\\
 \label{Aq2}
 \norm{uv}_{X_\pm^{-1-\sigma, \frac12+\sigma}}&\lesssim  \norm{u}_{X_{\pm_1}^{1-\sigma, 1-\sigma}}
 \norm{v}_{X_{\pm_2}^{-\sigma, 1-\sigma}}.
\end{align}
We would like to reduce \eqref{Aq1}--\eqref{Aq2} to the corresponding estimates 
where all the $X$ norms are replaced by $H$ norms so that 
we can apply the product estimates in Theorem \ref{ThmAtlas} to prove them.  In view of \eqref{HXH}, 
we can do this to the right hand side, but not to the left hand side since the 
hyperbolic exponents to the $X$ norms are positive, i.e., $ \frac12+\sigma$. We can nevertheless do the following trick 
to fix the problem.
Let $(\lambda, \eta)$ and $(\tau-\lambda, \xi-\eta)$ be the Fourier variables for $u$ and $v$ 
respectively, such that $(\tau, \xi)$ is the Fourier variable for the product $uv$. Then we have the identity
$$
-\tau \pm |\xi|= (-\lambda \pm_1 |\eta|) +  (-(\tau-\lambda) \pm_2 |\xi-\eta|) + (\pm |\xi| \mp_1  |\eta|\mp_2 |\xi-\eta| ), 
$$
which implies 
\begin{equation}\label{HypW}
 |-\tau \pm |\xi| |\lesssim |-\lambda \pm_1 |\eta| | +   |-(\tau-\lambda) \pm_2 |\xi-\eta|| + \max( |\eta|, |\xi-\eta| ). 
\end{equation}
Now, using this estimate we can reduce \eqref{Aq1} (here we also use symmetry)
\begin{equation*}
\left\{
\begin{aligned} 
\norm{uv}_{H^{-\sigma, 0}}&\lesssim  \norm{u}_{H^{1-\sigma, \frac12-2\sigma}}
 \norm{v}_{ H^{1-\sigma, 1-\sigma}},
 \\
 \norm{uv}_{H^{-\sigma, 0}}&\lesssim  \norm{u}_{H^{ \frac12-2\sigma, 1-\sigma}}
 \norm{v}_{ H^{1-\sigma, 1-\sigma}}
\end{aligned}
\right.
\end{equation*} 
both of which hold by Theorem \ref{ThmAtlas}.
Similarly, \eqref{Aq2} reduces to
\begin{equation*}
\left\{
\begin{aligned} 
\norm{uv}_{H^{-1-\sigma, 0}}&\lesssim  \norm{u}_{H^{1-\sigma, \frac12-2\sigma}}
 \norm{v}_{ H^{-\sigma, 1-\sigma}}
 \\
 \norm{uv}_{H^{-1-\sigma, 0}}&\lesssim  \norm{u}_{H^{1-\sigma, 1-\sigma}}
 \norm{v}_{ H^{-\sigma, \frac12-2\sigma}}
 \\
 \norm{uv}_{H^{-1-\sigma, 0}}&\lesssim  \norm{u}_{H^{\frac12-2\sigma, 1-\sigma}}
 \norm{v}_{ H^{-\sigma, 1-\sigma}}
 \\
 \norm{uv}_{H^{-1-\sigma, 0}}&\lesssim  \norm{u}_{H^{1-\sigma, 1-\sigma}}
 \norm{v}_{ H^{-\frac12-2\sigma,1-\sigma}},
\end{aligned}
\right.
\end{equation*} 
all of which hold by Theorem \ref{ThmAtlas}.

Finally, we prove \eqref{Al}. Clearly, $\partial_j A_\pm \in X_\pm^{-\sigma, \frac12+\sigma}(S_T) $ by \eqref{AX}.
We remain to prove 
 $$\partial_t A_\pm \in X_\pm^{-\sigma, \frac12+\sigma}(S_T) .$$ 
 To do this, we first use the equation \eqref{Schem} to write
$$
\partial_t A_\pm=\pm i\angles{\nabla} A_\pm \pm i(2\angles{\nabla})^{-1}M'(A_+, A_-).
$$
The term $\angles{\nabla} A_\pm$ is clearly in $X_\pm^{-\sigma, \frac12+\sigma}(S_T)$. Looking at the terms in 
$M'$, we then reduce to proving 
\begin{align*}
  A, \  A\partial A , A^3 \in X_\pm^{-1-\sigma, \frac12+\sigma}(S_T).
  \end{align*}

For the linear
term $A=A_+ + A_-$, it suffices to show 
$A_\mp \in X_\pm^{-1-\sigma, \frac12+\sigma}(S_T) $. But using the the simple inequality 
$\angles{-\tau\pm |\xi|}\lesssim \angles{-\tau\mp |\xi|}\angles{\xi} $, this reduces to 
proving $A_\mp \in X_\mp^{-\frac12, \frac12+\sigma}(S_T) $,  which holds by \eqref{AX}.
By \eqref{Aq} the bilinear term $ A\partial A$ is also in  $X_\pm^{-1-\sigma, \frac12+\sigma}(S_T)$.

It remains to prove that the cubic term $A^3\in X_\pm^{-1-\sigma, \frac12+\sigma}(S_T)$. Since $A^2 \in X_\pm^{-\sigma, \frac12+\sigma}(S_T)$ by \eqref{Aq}, 
we reduce to 
\begin{align*}
\label{Aq3}
 \norm{uv}_{X_\pm^{-1-\sigma, \frac12+\sigma}}&\lesssim  \norm{u}_{X_{\pm_1}^{1-\sigma, 1-\sigma}}
 \norm{v}_{X_{\pm_2}^{-\sigma, \frac12+\sigma}}.
\end{align*}
 Again using \eqref{HypW}, this reduces to
\begin{equation*}
\left\{
\begin{aligned} 
\norm{uv}_{H^{-1-\sigma, 0}}&\lesssim  \norm{u}_{H^{1-\sigma, \frac12-2\sigma}}
 \norm{v}_{ H^{-\sigma, \frac12+ \sigma}},
 \\
 \norm{uv}_{H^{-1-\sigma, 0}}&\lesssim  \norm{u}_{H^{1-\sigma, 1-\sigma}}
 \norm{v}_{ H^{-\sigma, 0}},
 \\
 \norm{uv}_{H^{-1-\sigma, 0}}&\lesssim  \norm{u}_{H^{\frac12-2\sigma, 1-\sigma}}
 \norm{v}_{ H^{-\sigma, \frac12+\sigma}},
 \\
 \norm{uv}_{H^{-1-\sigma, 0}}&\lesssim  \norm{u}_{H^{1-\sigma, 1-\sigma}}
 \norm{v}_{ H^{-\frac12-2\sigma, \frac12+\sigma}}
 \end{aligned}
\right.
\end{equation*} 
all of which hold by Theorem \ref{ThmAtlas}.

\subsection*{Acknowledgement}
I would like to thank Sebastian Herr for suggesting me to work on the low regularity problem of YM and for having 
useful discussions. I would also like to thank Sigmund Selberg for his encouragement, useful discussions
and for his hospitality when I visited the University of Bergen.

\end{document}